\documentclass[10pt]{article}
\usepackage[a4paper]{geometry}
\usepackage[utf8]{inputenc}
\usepackage[T1]{fontenc}
\usepackage{lmodern}
\usepackage[hidelinks]{hyperref}
\usepackage{url}
\usepackage{booktabs, bookmark}
\usepackage{amsfonts}
\usepackage{nicefrac}
\usepackage{microtype}
\usepackage{graphicx}
\usepackage{multirow}
\usepackage{amsmath,amsthm,amssymb}
\usepackage{mathrsfs}
\usepackage[title]{appendix}
\usepackage{xcolor}
\usepackage{textcomp}
\usepackage{manyfoot}
\usepackage{algorithm,algorithmic}
\usepackage{listings}
\usepackage{bm}
\usepackage{float,needspace}
\usepackage{empheq}
\usepackage{mathtools}
\usepackage{longtable}
\usepackage{enumitem}

\graphicspath{{./images/}}

\theoremstyle{plain}
\newtheorem{theorem}{Theorem}
\newtheorem{lemma}{Lemma}

\theoremstyle{definition}
\newtheorem{definition}{Definition}

\newtheorem{assumption}{Assumption}

\theoremstyle{remark}
\newtheorem{remark}{Remark}

\newcommand{\dist}{\mathrm{dist}}

\newcommand{\mcA}{\mathcal{A}}
\newcommand{\mcC}{\mathcal{C}}

\DeclareMathOperator{\argmin}{arg\,min}

\newcommand{\R}{\mathbb R}
\newcommand{\ip}[2]{\langle #1,#2\rangle}
\newcommand{\norm}[1]{\|#1\|}
\DeclareMathOperator{\dom}{dom}
\DeclareMathOperator{\prox}{prox}
\setlist[enumerate]{label=(\alph*),leftmargin=2em,itemsep=2pt,topsep=3pt}
\allowdisplaybreaks[2]

\title{Efficient Parameter-Free First-Order Methods
for Nonsmooth Composite Minimax Optimization}
\author{
Shaozhe Ke
\thanks{
Department of Industrial Engineering and Decision Analytics, Hong Kong University of Science and Technology, Hong Kong, China (email: {\tt shaozheke1220@gmail.com, symei@ust.hk, jiheng@ust.hk}).
}
\and
Sanyou Mei
\footnotemark[1]
\and
Jiheng Zhang
\footnotemark[1]}

\date{September 22, 2026}
\begin{document}
\maketitle\vspace{-2em}
\begin{abstract}
In this paper we propose first-order methods for a class of nonsmooth composite strongly convex--strongly concave and nonconvex--concave minimax optimization. We first develop an inexact proximal method and an accumulative regularizated method for strongly convex--strongly concave problems. The latter achieves the optimal dependence on the curvature parameters and places the smaller curvature inside the accuracy logarithm. Using this method as a subsolver, we propose a proximal point method for nonconvex--concave problems. Under suitable assumptions, it finds an $\epsilon$-stationary point with an operation complexity of $O(\epsilon^{-5/2})$, which improves the best-known $O(\epsilon^{-5/2}\log(1/\epsilon))$ bounds by removing the logarithmic factor. We further develop parameter-free variants for both problem classes, and achieve the same complexity without knowledge of any problem constants. All proposed methods are equipped with verifiable termination criteria.
\end{abstract}

\noindent\textbf{Keywords:}
Minimax optimization, parameter free, first-order methods, accelerated proximal methods,
accumulated regularization, operation complexity.\\

\noindent\textbf{Mathematics Subject Classification:}
90C47, 90C26, 90C25, 65K05.

\section{Introduction}
In this paper, we consider the composite minimax problem
\begin{equation}\label{eq:ncc}
\min_x\max_y\{H(x,y):=h(x,y)+p(x)-q(y)\}.
\end{equation}
Throughout the paper, we let $\mathcal X:=\dom p$ and $\mathcal Y:=\dom q$. We make the following assumptions.
\begin{assumption}\label{ass:ncc}
\begin{enumerate}[label=(\roman*)]
\item $p:\R^m\to(-\infty,+\infty]$ and $q:\R^n\to(-\infty,+\infty]$ are proper closed convex functions, and the proximal operator of $p$ and $q$ can be exactly evaluated.
\item The set $\mathcal Y$ is bounded, namely, we have $\|y\|\leq D_{\mathbf y}$ for some $D_{\mathbf y}\geq1$.
\item $h:\R^{m+n}\to\R$ is $L_{\nabla h}$-smooth on $\mathcal X\times\mathcal Y$ for some $L_{\nabla h}>0$. For any $x\in\mathcal X$, $h(x,\cdot)$ is concave on $\mathcal Y$.
\item Problem~\eqref{eq:ncc} has at least one optimal solution with finite optimal value $H^*=\min_x\max_y H(x,y)$.
\end{enumerate}
\end{assumption}

Problem~\eqref{eq:ncc} provides a common formulation for several learning and decision problems. Examples include distributionally robust learning~\cite{namkoongduchi2016,duchinamkoong2021,rafique2022} and learning with nondecomposable losses, including area-under-the-curve maximization~\cite{rafique2022,ying2016}. In these applications, the minimization variable represents a prediction model, while the maximization variable describes an uncertain distribution or arises from a variational representation of the loss. The convex terms $p$ and $q$ incorporate constraints and nonsmooth regularization. Problems of this form also arise as subproblems in methods for constrained minimax optimization~\cite{lumei2025minimax}. Their scale and composite structure motivate methods based on evaluations of $h$, $\nabla h$, and the proximal operators of $p$ and $q$.

There have been numerous studies on structured special cases of~\eqref{eq:ncc}. When $h(\cdot,y)$ is convex for every $y\in\mathcal Y$, its saddle-point optimality conditions form a monotone inclusion (MI). More generally, consider finding $u$ such that $0\in(F+B)(u)$, where $F$ is monotone and $L_F$-Lipschitz continuous, and $B$ is maximally monotone with an exactly computable resolvent. Proximal-point, splitting, extragradient, and extrapolation methods have been extensively studied for monotone inclusions and their saddle-point special cases~\cite{rockafellar1976,tseng2000,nemirovski2004,malitskytam2020,mokhtari2020,lumei2024}. More recent work establishes accelerated pointwise residual guarantees through anchoring and related extrapolation schemes~\cite{diakonikolas2020,yoonryu2021,caizheng2023,yoonryu2025}. Under $\mu$-strong monotonicity, an $O((L_F/\mu)\log(1/\epsilon))$ bound, with the remaining problem data fixed, can be obtained for the numbers of evaluations of $F$ and the resolvent of $B$, respectively, to find an $\epsilon$-residual solution~\cite{lumei2024}. Such methods apply to strongly convex--strongly concave (SCSC) minimax problems, but treating these problems as general strongly monotone inclusions only uses the smaller of their two curvature parameters.

More specifically, consider an SCSC problem with the same composite structure as~\eqref{eq:ncc}, with $h$ replaced by an $L_{\nabla f}$-smooth function $f$ that is $\mu_x$-strongly convex in $x$ and $\mu_y$-strongly concave in $y$, where $0<\mu_y\le\mu_x\le L_{\nabla f}$. Neither $\mathcal X$ nor $\mathcal Y$ needs to be bounded. A direct application of the above MI complexity has a dependence on $L_{\nabla f}/\mu_y$. Accelerated methods exploit the separate primal and dual curvatures to improve this dependence~\cite{lin2020near,wangli2020,jinsidfordtian2022,kovalev2022}, while lower bounds characterize the attainable dependence on the smoothness and curvature parameters~\cite{zhanghongzhang2022}. In particular, the method in~\cite{kovalev2022} attains the optimal factor $L_{\nabla f}/\sqrt{\mu_x\mu_y}$ by combining a pointwise-conjugate reformulation with acceleration and an MI subsolver. A modification with a verifiable termination criterion for approximate primal-dual stationarity was developed in~\cite{lumei2025minimax}. Nevertheless, when SCSC problems serve as subproblems and $\mu_y$ decreases with the target accuracy, the scale inside the accuracy logarithm also matters. 

We address this issue with two SCSC methods. The first applies an accelerated inexact proximal scheme to a strongly convex reformulation of the saddle problem. It builds on the reformulation in~\cite{kovalev2022}, proximal acceleration~\cite{guler1992,monteirosvaiter2013,linmairalharchaoui2018}, and relative-error criteria for inexact proximal methods~\cite{solodovsvaiter1999,bellocruz2020inexact,barre2022}. The resulting subproblems have comparable primal and dual curvatures and can be solved using an MI subsolver that returns a computed pointwise residual. Our analysis requires only a residual bound proportional to $L_F\|u^0-u^*\|/N$, where $N$ bounds the separate evaluation counts. The second method combines this contraction with accumulated regularization. This approach is related to recursive regularization for gradient minimization~\cite{allenzhu2018} and, more directly, the accumulated-regularization framework in~\cite{lan2024optimal}. By increasing the regularization and updating its center to retain the preceding quadratic terms, we obtain a bound with the smaller curvature $\mu_y$ inside the accuracy logarithm. This improvement is essential for the nonconvex--concave application.

When $h(\cdot,y)$ is nonconvex, most existing methods aim to find an approximate stationary point. This setting has been studied through multistep and two-timescale gradient descent-ascent~\cite{nouiehed2019,lin2020gda}, proximal-point methods~\cite{rafique2022,lin2020near,thekumparampil2019,kong2021,ostrovskii2021}, and smoothing methods~\cite{zhang2020smoothed,zhao2024smoothing,li2026smoothing}. Alternating proximal and projected-gradient methods have also been developed for composite and constrained formulations~\cite{botbohm2023,xu2023agp}. The intermediate nonconvex--strongly concave setting has also received substantial attention, including near-optimal algorithms and lower bounds~\cite{lin2020near,zhang2021complexity}. Two commonly studied criteria in the nonconvex--concave setting are stationarity of the primal value function $\max_y H(\cdot,y)$, measured through its Moreau envelope, and approximate satisfaction of the first-order conditions in both variables. The latter is the $\epsilon$-stationarity criterion used in this paper and is defined in Subsection~1.1. These criteria and the oracles used to attain them must be distinguished when comparing complexity bounds.

For stationarity of the primal value function, the method in~\cite{lin2020near} achieves an $O(\epsilon^{-3}\log^2(1/\epsilon))$ complexity for smooth constrained problems. For stationarity in both variables, the accelerated inexact proximal point smoothing method in~\cite{kong2021} achieves an $O(\epsilon^{-5/2})$ bound under an oracle that exactly solves a regularized maximization problem involving the coupling function. Such an oracle generally requires more than a gradient or a proximal evaluation. Using gradients and the separate proximal operators, an $O(\epsilon^{-5/2}\log(1/\epsilon))$ bound was established in~\cite{lumei2025minimax}; see also~\cite{li2026smoothing} for a smoothing and perturbation approach with the same dependence on $\epsilon$ for game stationarity. Removing this logarithmic factor while retaining computable stopping criteria is the main objective of our nonconvex analysis.

We propose an inexact proximal point method that uses the new SCSC solver to attain this objective. The smaller-curvature residual bound allows the total inner cost to be controlled by the descent over the outer iterations. Consequently, the method finds an $\epsilon$-stationary point with an operation complexity $O(\epsilon^{-5/2})$ . To the best of our knowledge, this paper is the first to remove the logarithmic factor for the general problem class in~\eqref{eq:ncc}. Independent and concurrent work~\cite{panzhengli2026} also removes this factor for a special case of~\eqref{eq:ncc} using a substantially different approach. Moreover, our framework accommodates proximable nonsmooth terms and admits a parameter-free variant that achieves the same complexity bound.

Implementing minimax methods often requires smoothness and curvature parameters or a bound on the dual domain. Adaptive methods have been studied for monotone inclusions and convex optimization~\cite{lumei2024,diakonikolas2020,lan2024optimal,lumei2023apg}. Methods that adapt to unknown smoothness and strong concavity parameters have also been developed for nonconvex--strongly concave problems~\cite{yang2022neada,li2023tiada}, and parameter-free line-search methods have been proposed for the nonconvex--concave setting~\cite{yang2025parameterfree}. We develop parameter-free variants for both problem classes considered here. The SCSC method accepts tentative parameters through a computed residual-reduction test. The nonconvex--concave method adapts its smoothness estimate through a verifiable model inequality, and its parameter-free variant also selects a dual-radius estimate. Their complexity bounds count unsuccessful trials as well as successful ones. The main contributions are summarized below.
\begin{itemize}
\item We develop an accelerated inexact proximal method that reduces the solution distance by a fixed factor within $O(L_{\nabla f}/\sqrt{\mu_x\mu_y})$ evaluations of $\nabla f$, $\prox_p$, and $\prox_q$, respectively. Using it as a subsolver, our accumulative regularization method finds an $\epsilon$-stationary point within $O((L_{\nabla f}/\sqrt{\mu_x\mu_y})\log(1/\epsilon))$ evaluations of each oracle. This bound retains the optimal dependence on the curvature parameters and places the smaller curvature $\mu_y$ inside the accuracy logarithm, which is essential for removing the logarithmic factor in the nonconvex--concave complexity. We also propose a parameter-free variant that preserves the optimal complexity for SCSC minimax problem. All proposed algorithms are equipped with verifiable termination criteria.

\item We develop a proximal point method that finds an $\epsilon$-stationary point of~\eqref{eq:ncc} within $O(\epsilon^{-5/2})$ evaluations of $h$, $\nabla h$, $\prox_p$, and $\prox_q$, respectively, for fixed problem data. To the best of our knowledge, this paper is the first to remove the logarithmic factor from the best-known $O(\epsilon^{-5/2}\log(1/\epsilon))$ complexity bound for the general problem class~\eqref{eq:ncc}, while the independent and concurrent work~\cite{panzhengli2026} obtains the same improvement for a special case of~\eqref{eq:ncc} using a substantially different approach.
 Moreover, our method is parameter-free and equipped with a verifiable termination criterion.

\end{itemize}

The rest of this paper is organized as follows. In Subsection~1.1, we introduce some notation and terminology. In Section~2, we propose first-order methods for SCSC minimax optimization and study their complexity. In Section~3, we propose a first-order method for problem~\eqref{eq:ncc} and its parameter-free variant, and present complexity results for them. Finally, we provide the proofs of the main results in Section~4.

\subsection{Notation and preliminaries}
The following notation will be used throughout this paper. Let $\R^d$ denote the Euclidean space of dimension $d$, $\ip{\cdot}{\cdot}$ denote the standard inner product, and $\norm{\cdot}$ denote the Euclidean norm. We write $\mathbb N=\{1,2,\ldots\}$. For any $a\in\R$, let $a_+=\max\{a,0\}$, and let $\lceil a\rceil$ and $\lfloor a\rfloor$ denote the least integer greater than or equal to $a$ and the greatest integer less than or equal to $a$, respectively. The notation $\log$ denotes the natural logarithm, and $\log_b$ denotes the logarithm to base $b$. Given a point $u\in\R^d$ and a set $S\subseteq\R^d$, let $\dist(u,S)=\inf_{v\in S}\norm{u-v}$, with $\dist(u,\varnothing)=+\infty$.

Given a proper closed convex function $g:\R^d\to(-\infty,+\infty]$, its domain is $\dom g=\{u:g(u)<+\infty\}$, and its subdifferential is
\[
\partial g(u)=\{v:g(z)\ge g(u)+\ip{v}{z-u}\ \text{for all }z\in\R^d\}\qquad\forall u\in\dom g,
\]
with $\partial g(u)=\varnothing$ for $u\notin\dom g$. The proximal operator associated with $g$ is denoted by $\prox_g$ and defined as
\[
\prox_g(u)=\argmin_{z\in\R^d}\left\{g(z)+\frac12\norm{z-u}^2\right\}\qquad\forall u\in\R^d.
\]
We count the evaluation of $\prox_{\gamma g}(u)$ as one evaluation of the proximal operator of $g$ for any $\gamma>0$ and $u\in\R^d$.

A mapping $G$ is said to be $L_G$-Lipschitz continuous on a set $S$ if $\norm{G(u)-G(v)}\le L_G\norm{u-v}$ for all $u,v\in S$. A function $g$ is said to be $L_{\nabla g}$-smooth on $S$ if $\nabla g$ exists on $S$ and is $L_{\nabla g}$-Lipschitz continuous there. A function $g$ is $\mu$-strongly convex on a convex set $S$ if $g-\mu\norm{\cdot}^2/2$ is convex on $S$, and $L$-weakly convex on $S$ if $g+L\norm{\cdot}^2/2$ is convex on $S$. It is $\mu$-strongly concave if $-g$ is $\mu$-strongly convex.

\begin{definition}\label{def:epsilon-stationary}
For $\epsilon\ge0$, a point $(x,y)\in\mathcal X\times\mathcal Y$ is called an $\epsilon$-stationary point of the minimax problem~\eqref{eq:ncc} if
\[
\dist(0,\nabla_x h(x,y)+\partial p(x))\le\epsilon,\qquad
\dist(0,\nabla_y h(x,y)-\partial q(y))\le\epsilon.
\]
We use the same definition for other composite minimax problems, with $h$ replaced by the corresponding smooth coupling.
\end{definition}

For a set-valued operator $T:\R^d\rightrightarrows\R^d$, let $\dom T=\{u:T(u)\ne\varnothing\}$. The operator $T$ is $\mu$-strongly monotone if
\[
\ip{v-w}{u-z}\ge\mu\norm{u-z}^2\qquad\forall u,z\in\dom T,\ v\in T(u),\ w\in T(z),
\]
and monotone if the same inequality holds with $\mu=0$. It is maximally monotone if its graph is not properly contained in the graph of another monotone operator. Let $I$ denote the identity operator. For maximally monotone $T$ and $\gamma>0$, the resolvent $(I+\gamma T)^{-1}$ is single-valued, defined on $\R^d$, and nonexpansive, that is, $1$-Lipschitz continuous. We count the evaluation of $(I+\gamma T)^{-1}(u)$ as one evaluation of the resolvent of $T$ for any $\gamma>0$ and $u\in\R^d$.

Define
\[
A(x,y)=\begin{pmatrix}\nabla_x h(x,y)\\-\nabla_y h(x,y)\end{pmatrix},\qquad
B(x,y)=\begin{pmatrix}\partial p(x)\\\partial q(y)\end{pmatrix}.
\]
The KKT system of~\eqref{eq:ncc} is $0\in(A+B)(u)$. By Assumption~\ref{ass:ncc}(iii), $A$ is $L_{\nabla h}$-Lipschitz continuous on $\mathcal X\times\mathcal Y$, and $A+L_{\nabla h}I$ is monotone there. By Assumption~\ref{ass:ncc}(i), $B$ is maximally monotone, and for every $\gamma>0$,
\[
(I+\gamma B)^{-1}(x,y)=(\prox_{\gamma p}(x),\prox_{\gamma q}(y)).
\]

\section{First-order methods for strongly convex--strongly concave minimax optimization}
\label{sec:algorithm1}

In this section we consider the strongly convex--strongly concave minimax problem
\begin{equation}\label{eq:scsc}
\min_x\max_y\{f(x,y)+p(x)-q(y)\},
\end{equation}
where $p:\R^m\to(-\infty,+\infty]$ and $q:\R^n\to(-\infty,+\infty]$ are proper, closed, convex, and proximable. No boundedness assumption is imposed on their domains. We make the following assumption on $f$.

\begin{assumption}\label{ass:scsc}
The function $f$ is $\mu_x$-strongly convex in $x$, $\mu_y$-strongly concave in $y$, and $L_{\nabla f}$-smooth on $\mathcal X\times\mathcal Y$, where $0<\mu_y\le\mu_x\le L_{\nabla f}$.
\end{assumption}
\begin{remark}
The ordering $\mu_y\le\mu_x$ entails no loss of generality. If $\mu_x<\mu_y$, exchanging the two variable blocks and reversing the objective sign yields the equivalent saddle-point with the curvature bounds and proximal mappings interchanged.
\end{remark}

Under Assumption~\ref{ass:scsc}, problem~\eqref{eq:scsc} has a unique saddle point. For $u=(x,y)$, set
\begin{equation}\label{eq:scsc-operators}
F(u)=\begin{pmatrix}\nabla_x f(x,y)\\-\nabla_y f(x,y)\end{pmatrix},\qquad B(u)=\begin{pmatrix}\partial p(x)\\\partial q(y)\end{pmatrix}.
\end{equation}
The KKT system of~\eqref{eq:scsc} is $0\in(F+B)(u)$. Under Assumption~\ref{ass:scsc}, $F$ is $L_{\nabla f}$-Lipschitz and $\mu_y$-strongly monotone on $\mathcal X\times\mathcal Y$. Moreover, $B$ is maximally monotone, and its resolvent is evaluated by the proximal operators of $p$ and $q$.

More generally, consider monotone inclusion $0\in(F+B)(u)$, where $B$ is maximally monotone with an exactly computable resolvent, $F$ is monotone and $L_F$-Lipschitz on $\dom B$, and $F+B$ is $\mu$-strongly monotone on $\dom B$, with $0<\mu\le L_F$. Suppose this inclusion has a solution $u^*$, which is unique. Although linearly convergent methods are available for strongly monotone inclusions~\cite{lumei2024}, a sublinear pointwise residual bound of order $L_F\norm{u^0-u^*}/(k+1)$ suffices for our analysis. We therefore impose the following requirement on the subsolver.

\begin{assumption}\label{ass:subsolver}
There exist an algorithm $\mcA(F,B,u^0,L_F,\mu)$ and constants $\mathcal C_{\mcA}>0$ and $N_{\mcA}\in\mathbb N$, with $N_{\mcA}\ge1$, such that the following holds. From any $u^0\in\dom B$, one run of $\mcA$ generates, for every $\vartheta>0$, a computed pair $(u,r)$ with
\[
r\in(F+B)(u),\qquad \norm r\le\vartheta\norm{u^0-u^*},
\]
using at most $N_{\mcA}+\lceil\mathcal C_{\mcA}L_F/\vartheta\rceil$ evaluations of $F$ and of the resolvent of $B$, respectively. Every evaluation of $F$ is at a point of $\dom B$. Both counts include initialization and computation of $r$. The constants depend only on $\mcA$, not on the inclusion, its parameters, or its initial point. Neither $u^*$ nor its distance from $u^0$ is an input.
\end{assumption}

\begin{remark}
For the subdifferential operator $B$ used here, the EAG and past-EAG methods in~\cite[Corollary~3.1]{trandinhnguyentrung2025}, with one forward--backward initialization step, satisfy Assumption~\ref{ass:subsolver}. With $\tau_k=1/(k+2)$, one may take $(N_{\mcA},\mathcal C_{\mcA})=(4,22)$ for EAG with stepsize $1/(2L_F)$, and $(N_{\mcA},\mathcal C_{\mcA})=(4,28)$ for past-EAG with stepsize $1/(3L_F)$.
\end{remark}
\subsection{An inexact proximal method for SCSC problems}
\label{sec:contract}

We first develop an accelerated inexact proximal method for the SCSC problem~\eqref{eq:scsc}. Our approach applies the accelerated proximal point framework of~\cite{bellocruz2020inexact} to a strongly convex reformulation of the saddle-point problem. Specifically, we consider
\begin{equation}\label{eq:convex-reformulation}
\min_{\xi,y}\Phi(\xi,y)\quad\mbox{where}\quad
\Phi(\xi,y)=q(y)+\sup_x \left\{\frac{\mu_x}{2}\norm{x-\xi}^2-f(x,y)-p(x)\right\} \qquad\forall y\in\mathcal Y,
\end{equation}
with $\Phi(\xi,y)=+\infty$ for $y\notin\mathcal Y$. The function $\Phi$ is proper, closed, and $\mu_y$-strongly convex, and its unique minimizer coincides with the saddle point $u^*$ of~\eqref{eq:scsc}. The key to implementing this approach is that the proximal step
\[
\min_{v\in\R^{m+n}} \left\{\Phi(v)+\frac{\mu_x}{2}\norm{v-c^k}^2\right\}
\]
can be computed through the balanced saddle problem
\begin{equation}\label{eq:bld}
\min_x\max_y\left\{f(x,y)+p(x)-q(y) -\frac{\mu_x}{4}\norm{x-c_x^k}^2 -\frac{\mu_x}{2}\norm{y-c_y^k}^2\right\}.
\end{equation}
For $\mu_y\leq\mu_x$, the quadratic terms balance the primal and dual curvatures at order $\mu_x$. Moreover, an approximate KKT pair for this problem yields, after a simple correction, an inexact proximal certificate for $\Phi$; see Lemma~\ref{lem:certificate}.

We control the accuracy of the inner solves through a relative-error criterion with an additional absolute allowance, based on the mixed-error setting studied in~\cite{barre2022}. Combining these inexact proximal steps with residual-corrected extrapolation yields Algorithm~\ref{alg:contract} below. The algorithm requires neither evaluating $\Phi$ nor computing the supremum defining it.

\begin{algorithm}[H]
\caption{An inexact proximal method for SCSC contraction}
\label{alg:contract}
\begin{algorithmic}[1]
\REQUIRE Subroutine $\mcA$ satisfying Assumption~\ref{ass:subsolver}, $L_{\nabla f},\mu_x,\mu_y$ satisfying Assumption~\ref{ass:scsc}, $\tilde u^0=c^0=u^0\in\mathcal X\times\mathcal Y$.
\FOR{$k=0,1,\ldots$}
\STATE Call $\mcA(F_k,B,\widehat c^k,L_{\nabla f}+\mu_x,\mu_x/2)$, where
\begin{equation}\label{eq:F-k}
F_k(x,y)=\begin{pmatrix} \nabla_x f(x,y)-\dfrac{\mu_x}{2}(x-c_x^k)\\[1mm]
-\nabla_y f(x,y)+\mu_x(y-c_y^k)
\end{pmatrix}\qquad\forall(x,y)\in\mathcal X\times\mathcal Y,
\end{equation}
\[
\widehat c^k=(I+\gamma_kB)^{-1} \bigl(c^k-\gamma_kF_k(u^0)\bigr),\qquad \gamma_k=\frac{1}{(L_{\nabla f}+\mu_x)(k+2)^2}.
\]
Terminate this call at a computed pair $u^{k+1}=(x^{k+1},y^{k+1})$, $r^{k+1}=(r_x^{k+1},r_y^{k+1})\in(F_k+B)(u^{k+1})$ satisfying
\begin{equation}\label{eq:relative}
\norm{r^{k+1}}\le\frac{\mu_x}{64} \left(\norm{u^{k+1}-c^k}+\frac{\norm{c^k-u^0}}{(k+2)^2}\right).
\end{equation}
\STATE Compute
\begin{align}
&e^{k+1}=(e_x^{k+1},e_y^{k+1})=F(u^{k+1})+r^{k+1}-F_k(u^{k+1}),\label{eq:algorithm1-updates}\\
&\tilde u^{k+1}=u^{k+1}-\frac1{\mu_x}(e_x^{k+1},0),\label{eq:algorithm1-corrected-point}\\
&c^{k+1}=\tilde u^{k+1}+\frac{k}{k+3}(\tilde u^{k+1}-\tilde u^k)+\frac{k+2}{(k+3)\mu_x}(2r_x^{k+1},-r_y^{k+1}).\label{eq:algorithm1-center}
\end{align}
\STATE Terminate and output $(u^{k+1},e^{k+1})$ if
\begin{equation}\label{eq:contract-stop}
\norm{e^{k+1}}\le\mu_x \left(\frac19-\frac{3}{k+2}\sqrt{\frac{\mu_x}{\mu_y}}\right)_+ \norm{u^{k+1}-u^0}.
\end{equation}
\ENDFOR
\end{algorithmic}
\end{algorithm}

\begin{remark}
It can be observed from~\eqref{eq:algorithm1-updates}--\eqref{eq:algorithm1-center} that $e^{k+1}$ is the original KKT residual, $\tilde u^{k+1}$ is an inexact proximal point for $\Phi$, and $c^{k+1}$ is the next extrapolation center. Moreover, the inner test~\eqref{eq:relative} controls the proximal inexactness, while~\eqref{eq:contract-stop} determines when the algorithm returns. The decreasing parameter $\gamma_k$ is used only in the initialization resolvent.
\end{remark}

The following theorem presents our results on the above algorithm, whose proof is deferred to Subsection \ref{subsec:proof-algorithm1}.
\begin{theorem}\label{thm:contract}
Suppose Assumptions~\ref{ass:scsc} and~\ref{ass:subsolver} hold. Then Algorithm~\ref{alg:contract} is well-defined and terminates within $\lceil128\sqrt{\mu_x/\mu_y}\rceil$ iterations. If it terminates at iteration $k$, its output satisfies $e^{k+1}\in(F+B)(u^{k+1})$ and
\begin{equation}\label{eq:contract-guarantee}
\norm{u^{k+1}-u^*}\le\frac18\norm{u^0-u^*},\qquad\norm{e^{k+1}}\le\frac{\mu_x}{8}\norm{u^0-u^*}.
\end{equation}
The numbers of evaluations of $\nabla f$, $\prox_p$, and $\prox_q$ are bounded, respectively, by
\begin{equation}\label{eq:algorithm1-cost}
1+\left\lceil128\sqrt{\frac{\mu_x}{\mu_y}}\right\rceil\left(1+N_{\mcA}+\left\lceil\frac{128\mathcal C_{\mcA}(L_{\nabla f}+\mu_x)}{\mu_x}\right\rceil\right)=O\!\left(\frac{L_{\nabla f}}{\sqrt{\mu_x\mu_y}}\right).
\end{equation}
The hidden constant depends only on $N_{\mcA}$ and $\mathcal C_{\mcA}$.
\end{theorem}

The proof is given in Section~\ref{subsec:proof-algorithm1}. The iteration bound is used only in the analysis; the algorithm stops by~\eqref{eq:contract-stop}.
\subsection{An accumulative regularizated method for SCSC problems}
\label{sec:scsc}

In this subsection we use Algorithm~\ref{alg:contract} to find an $\epsilon$-stationary point of~\eqref{eq:scsc}, where $\epsilon>0$ is given. Algorithm~\ref{alg:contract} reduces the initial solution distance by an absolute factor. Its residual bound, however, involves $\mu_x$. We combine these contractions with accumulated regularization to obtain an evaluation bound involving the smaller curvature $\mu_y$ inside the accuracy logarithm.

At each regularized iteration we approximately solve a saddle problem with center $\bar u=(\bar u_x,\bar u_y)$ of the form
\[
\min_x\max_y\left\{f(x,y)+p(x)-q(y)+\frac\lambda2\norm{x-\bar u_x}^2-\frac\lambda2\norm{y-\bar u_y}^2\right\},\qquad\lambda>0.
\]
Following the accumulated-center construction in~\cite[Algorithm~2.1]{lan2024optimal}, we increase $\lambda$ and update $\bar u$ so that the preceding quadratic terms are retained. One preliminary call to Algorithm~\ref{alg:contract} bounds the remaining solution distance by its computed displacement. We use this displacement and $\epsilon$ to choose the initial regularization, and then continue the accumulation until the original residual is at most $\epsilon$.

\begin{algorithm}[H]
\caption{Am accumulated regularizated method for SCSC minimax optimization}
\label{alg:scsc}
\begin{algorithmic}[1]
\REQUIRE Subroutine $\mcA$ satisfying Assumption~\ref{ass:subsolver}, $L_{\nabla f},\mu_x,\mu_y$ satisfying Assumption~\ref{ass:scsc}, $u^0\in\mathcal X\times\mathcal Y$, $\epsilon>0$.
\STATE Call Algorithm~\ref{alg:contract} with $(L_{\nabla f},\mu_x,\mu_y,u^0,\mcA)\leftarrow (L_{\nabla f},\mu_x,\mu_y,u^0,\mcA)$ for solving problem~\eqref{eq:scsc}, and denote its output by $(u^1,e^1)$. Terminate and output $(u^1,e^1)$ if $\norm{e^1}\le\epsilon$. Otherwise, set $\bar u^1=u^1$ and $\lambda_1=\epsilon/\norm{u^1-u^0}$.\par
\FOR{$k=1,2,\ldots$}
\STATE Call Algorithm~\ref{alg:contract} with $(L_{\nabla f},\mu_x,\mu_y,u^0,\mcA)\leftarrow(L_{\nabla f}+\lambda_k,\mu_x+\lambda_k,\mu_y+\lambda_k,u^k,\mcA)$ for solving $\min_x\max_y\{f_k(x,y)+p(x)-q(y)\}$ and denote its output by $(u^{k+1},r^{k+1})$, where
\[
f_k(x,y)=f(x,y)+\frac{\lambda_k}{2}\norm{x-\bar u_x^k}^2-\frac{\lambda_k}{2}\norm{y-\bar u_y^k}^2.
\]
\STATE Compute $e^{k+1}=r^{k+1}-\lambda_k(u^{k+1}-\bar u^k)$.
\STATE Terminate and output $(u^{k+1},e^{k+1})$ if $\norm{e^{k+1}}\le\epsilon$. Otherwise, set $\lambda_{k+1}=4\lambda_k$ and $\bar u^{k+1}=(\bar u^k+3u^{k+1})/4$.
\ENDFOR
\end{algorithmic}
\end{algorithm}

\begin{remark}
The preliminary call determines $\lambda_1$ from the computed displacement. At each regularized iteration, $r^{k+1}$ is the residual of the regularized subproblem, and step~4 converts it to the original residual $e^{k+1}$ used in the stopping test.
\end{remark}

The following theorem presents our results on the above algorithm, whose proof is deferred to Subsection \ref{subsec:proof-algorithm2}.
\begin{theorem}\label{thm:scsc}
Suppose Assumptions~\ref{ass:scsc} and~\ref{ass:subsolver} hold. Then Algorithm~\ref{alg:scsc} is well-defined and terminates. Its output $(u^{k+1},e^{k+1})$, with $k=0$ for a return in step~1, satisfies $e^{k+1}\in(F+B)(u^{k+1})$ and $\norm{e^{k+1}}\le\epsilon$. The numbers of evaluations of $\nabla f$, $\prox_p$, and $\prox_q$ are bounded, respectively, by
\begin{equation}\label{eq:scsc-cost}
\begin{aligned}[b]
&2^{11}(N_{\mcA}+256\mathcal C_{\mcA}+3)\frac{L_{\nabla f}}{\sqrt{\mu_x\mu_y}}\left(1+\log_4\!\left(1+\frac{\mu_y\norm{u^0-u^*}}{\epsilon}\right)\right)\\
&=O\!\left(\frac{L_{\nabla f}}{\sqrt{\mu_x\mu_y}}\left(1+\log\!\left(1+\frac{\mu_y\norm{u^0-u^*}}{\epsilon}\right)\right)\right),
\end{aligned}
\end{equation}
where $u^*$ is the saddle point of~\eqref{eq:scsc} and the hidden constant depends only on $N_{\mcA}$ and $\mathcal C_{\mcA}$.
\end{theorem}

\subsection{A parameter-free method for SCSC problems}\label{sec:pf-scsc}

In this subsection we develop a variant of Algorithm~\ref{alg:scsc} for solving~\eqref{eq:scsc} without knowing $L_{\nabla f}$, $\mu_x$, or $\mu_y$. We call Algorithm~\ref{alg:scsc} with tentative smoothness and curvature parameters and accept its output when the residual is sufficiently reduced.

In addition to Assumption~\ref{ass:subsolver}, we require an implementation of $\mcA$ that is well-defined for all positive supplied parameters and can be interrupted before an evaluation of $F$ or the resolvent of $B$. These requirements apply to every supplied inclusion, even when its smoothness and curvature parameters are invalid.
No convergence guarantee is required for invalid parameters.

For the simplicity of our discussion, let $\mathcal I$ be the list of integer triples $(i,j,k)$ with $i\ge0$ and $0\le j\le k$, ordered by increasing $2i+j+k$, with any fixed order for ties.
Each call to Algorithm~\ref{alg:scsc} below is called a trial. Its evaluation limit applies separately to the numbers of evaluations of $\nabla f$, $\prox_p$, and $\prox_q$, including all inner calls, initializations, and residual computations. 

\begin{algorithm}[H]
\caption{A parameter-free method for SCSC minimax optimization}
\label{alg:pf-scsc}
\begin{algorithmic}[1]
\REQUIRE Subroutine $\mcA$ satisfying Assumption~\ref{ass:subsolver}, $u^0=(u_x^0,u_y^0)\in\mathcal X\times\mathcal Y$, $d\ne0$ with $u_x^0+d\in\mathcal X$, and $\epsilon>0$.
\STATE Set $\rho=\norm{\nabla_x f(u_x^0+d,u_y^0)-\nabla_x f(u_x^0,u_y^0)}/\norm d$.
\STATE Compute $u^1=(I+\rho^{-1}B)^{-1}\bigl(u^0-\rho^{-1}F(u^0)\bigr)$ and $r^1=\rho(u^0-u^1)+F(u^1)-F(u^0)$.
\STATE Terminate and output $u^1$ if $\norm{r^1}\le\epsilon$. Otherwise, set $(i_1,j_1,k_1)=(0,0,0)$.
\FOR{$t=1,2,\ldots$}
\STATE Set $U_t=\rho 2^{i_t}$, $\alpha_t=\rho 2^{-j_t}$, and $\nu_t=\rho 2^{-k_t}$.
\STATE Call Algorithm~\ref{alg:scsc} with $(L_{\nabla f},\mu_x,\mu_y,u^0,\epsilon,\mcA)\leftarrow(U_t,\alpha_t,\nu_t,u^t,\max\{\epsilon,\norm{r^t}/8\},\mcA)$ for solving problem~\eqref{eq:scsc}, allowing at most $\lceil\widetilde{\mcC}_{\mcA}U_t/\sqrt{\alpha_t\nu_t}\rceil$ evaluations of $\nabla f$, $\prox_p$, and $\prox_q$, respectively, where $\widetilde{\mcC}_{\mcA}=2^{13}(N_{\mcA}+256\mathcal C_{\mcA}+3)$. Stop the call and declare the trial unsuccessful before any evaluation exceeding its corresponding limit. If the call terminates with an output, denote it by $(\tilde u^{t+1},e^{t+1})$.
\STATE Terminate and output $\tilde u^{t+1}$ if the call returns a pair satisfying $\norm{e^{t+1}}\le\epsilon$.
\IF{the call returns a pair satisfying $\norm{e^{t+1}}\le\norm{r^t}/2$}
\STATE Set $(u^{t+1},r^{t+1})=(\tilde u^{t+1},e^{t+1})$ and $(i_{t+1},j_{t+1},k_{t+1})=(i_t,j_t,k_t)$.
\ELSE
\STATE Set $(u^{t+1},r^{t+1})=(u^t,r^t)$ and let $(i_{t+1},j_{t+1},k_{t+1})$ be the next triple after $(i_t,j_t,k_t)$ in $\mathcal I$.
\ENDIF
\ENDFOR
\end{algorithmic}
\end{algorithm}

\begin{remark}
For the update of $(i_{t+1},j_{t+1},k_{t+1})$ in step 11, one may choose $(i,j,k)$ lexicographically among triples with the same value of $2i+j+k$. The resulting list $\mathcal I$ starts with $(0,0,0),(0,0,1),(0,0,2),(0,1,1),(1,0,0),\ldots$. A successful trial retains the current triple, while an unsuccessful trial advances to the next triple and retains the current point and residual.
\end{remark}

The following theorem presents our results on the above algorithm, whose proof is deferred to Subsection \ref{subsec:proof-algorithm3}.
\begin{theorem}\label{thm:pf-scsc}
Suppose Assumption~\ref{ass:scsc} holds, and $\mcA$ satisfies Assumption~\ref{ass:subsolver} and the additional execution requirements in this subsection. Then Algorithm~\ref{alg:pf-scsc} is well-defined and terminates with an $\epsilon$-stationary point of~\eqref{eq:scsc} for every input satisfying its requirements. The numbers of evaluations of $\nabla f$, $\prox_p$, and $\prox_q$ are bounded, respectively, by
\begin{equation}\label{eq:pf-complexity}
\begin{aligned}[b]
&3+4(\widetilde{\mcC}_{\mcA}+1)\frac{L_{\nabla f}}{\sqrt{\mu_x\mu_y}}\left((2+\sqrt2)\left(2\log_2\frac{L_{\nabla f}}{\sqrt{\mu_x\mu_y}}+6\right)^2+\left\lceil\log_2\!\left(1+\frac{4L_{\nabla f}^2\norm{u^0-u^*}}{\mu_x\epsilon}\right)\right\rceil\right)\\
&=O\!\left(\frac{L_{\nabla f}}{\sqrt{\mu_x\mu_y}}\left(\log^2\!\left(2+\frac{L_{\nabla f}}{\sqrt{\mu_x\mu_y}}\right)+\log\!\left(1+\frac{\mu_y\norm{u^0-u^*}}\epsilon\right)\right)\right),
\end{aligned}
\end{equation}
where $u^*$ is the saddle point of~\eqref{eq:scsc} and $\widetilde{\mcC}_{\mcA}$ is defined in step~6 of Algorithm~\ref{alg:pf-scsc}.
\end{theorem}

\begin{remark}
\begin{enumerate}[label=(\roman*)]
\item A successful trial verifies residual reduction, not the validity of its global parameter estimates. The parameter search contributes the first term in~\eqref{eq:pf-complexity}, which is additive to the accuracy term.
\item The logarithmic term in~\eqref{eq:pf-complexity} satisfies
\[
\left\lceil\log_2\!\left(1+\frac{4L_{\nabla f}^2\norm{u^0-u^*}}{\mu_x\epsilon}\right)\right\rceil
\le3+2\log_2\frac{L_{\nabla f}}{\sqrt{\mu_x\mu_y}}+\log_2\!\left(1+\frac{\mu_y\norm{u^0-u^*}}\epsilon\right).
\]
Since $L_{\nabla f}/\sqrt{\mu_x\mu_y}\ge1$, the constant and the additional logarithm of $L_{\nabla f}/\sqrt{\mu_x\mu_y}$ are absorbed into the squared logarithm in~\eqref{eq:pf-complexity}. This yields the stated $O(\cdot)$ expression, whose hidden constant depends only on $N_{\mcA}$ and $\mathcal C_{\mcA}$.
\end{enumerate}
\end{remark}
\section{First-order method for nonconvex--concave minimax optimization}\label{sec:algorithms}

\subsection{A proximal point method for NCC problems without knowledge of $L_{\nabla L}$}\label{sec:ncc}

We first develop a method for~\eqref{eq:ncc} that uses a supplied radius $D\ge1$ and calls Algorithm~\ref{alg:scsc} without knowing $L_{\nabla h}$. By Assumption~\ref{ass:ncc}(ii), $\norm{y-y^0}\le2D_{\mathbf y}$ for every $y\in\mathcal Y$. Thus, with $D=2D_{\mathbf y}$, the method finds an $\epsilon$-stationary point of~\eqref{eq:ncc}. We impose the additional execution requirements on $\mcA$ stated in Subsection~\ref{sec:pf-scsc}, including when the supplied subproblems are not SCSC.

Each call to Algorithm~\ref{alg:scsc} is called a trial. The dual regularization parameter $\mu$ remains fixed. We maintain a primal center $x^t$, an initial pair $(\hat x,\hat y)$ for the inner call, and a residual $r^t$ of the subproblem~\eqref{eq:h-k} at that pair. After an accepted trial, the returned primal point becomes the next center and the returned pair initializes the next call. A rejected trial retains the center and initial pair and doubles $\beta_t$. A trial is called accepted if its call returns, passes the radius test in step~5, fails the stopping test in step~6, and satisfies~\eqref{eq:ncc-model-test}. The radius test returns failure when the supplied radius is exceeded and Algorithm~\ref{alg:pf-ncc} uses this outcome to select the radius.

\begin{algorithm}[H]
\caption{A proximal point method for nonconvex--concave minimax optimization without knowledge of $L_{\nabla h}$}
\label{alg:ncc}
\begin{algorithmic}[1]
\REQUIRE Subroutine $\mcA$ satisfying Assumption~\ref{ass:subsolver}, $(x^0,y^0)\in\mathcal X\times\mathcal Y$, $D\ge1$, and $\epsilon>0$.
\STATE Set $\beta_1=\epsilon/32$, $\mu=\epsilon/(32D)$, and $x^1=x^0$.
\STATE Compute 
\begin{align*}
&(\hat x,\hat y)=(I+\beta_1^{-1}B)^{-1}\bigl((x^0,y^0)-\beta_1^{-1}A(x^0,y^0)\bigr)\\
&r^1=A(\hat x,\hat y)-A(x^0,y^0)+\bigl(\beta_1(\hat x-x^0),(\mu-\beta_1)(\hat y-y^0)\bigr).
\end{align*}
\FOR{$t=1,2,\ldots$}
\STATE Set $\tau_t=(\epsilon/64)\sqrt{\mu/\beta_t}$. Call Algorithm~\ref{alg:scsc} with $(L_{\nabla f},\mu_x,\mu_y,u^0,\epsilon,\mcA)\leftarrow(3\beta_t,\beta_t,\mu,(\hat x,\hat y),\tau_t,\mcA)$ for solving $\min_x\max_y\{h_t(x,y)+p(x)-q(y)\}$, where
\begin{equation}\label{eq:h-k}
h_t(x,y)=h(x,y)+\beta_t\norm{x-x^t}^2-\frac{\mu}{2}\norm{y-y^0}^2,
\end{equation}
allowing at most
\begin{equation}\label{eq:ncc-trial-budget}
\left\lceil\widetilde{\mcC}_{\mcA}\sqrt{\frac{\beta_t}{\mu}}\left(1+\log_4\!\left(1+\frac{64}{\epsilon}\sqrt{\norm{r_x^t}^2+\frac{\beta_t}{\mu}\norm{r_y^t}^2}\right)\right)\right\rceil
\end{equation}
evaluations of $\nabla h$, $\prox_p$, and $\prox_q$, respectively, where $\widetilde{\mcC}_{\mcA}$ is defined in step~6 of Algorithm~\ref{alg:pf-scsc}. Stop the call before any evaluation exceeding its corresponding limit. If the call returns, denote its output by $((x^{t+1},y^{t+1}),e^{t+1})$.
\STATE If the call returns a pair satisfying $\norm{y^{t+1}-y^0}>D$, terminate and return failure.
\STATE If the call returns a pair satisfying $\norm{x^{t+1}-x^t}\le\epsilon/(4\beta_t)$, terminate and output $(x^{t+1},y^{t+1})$.
\STATE If the call returns a pair satisfying
\begin{equation}\label{eq:ncc-model-test}
h(x^t,y^{t+1})-h(x^{t+1},y^{t+1})+\ip{\nabla_x h(x^{t+1},y^{t+1})}{x^{t+1}-x^t}+\frac{\beta_t}{2}\norm{x^{t+1}-x^t}^2\ge0,
\end{equation}
set $(\hat x,\hat y)=(x^{t+1},y^{t+1})$, $r^{t+1}=e^{t+1}-\bigl(2\beta_t(x^{t+1}-x^t),0\bigr)$, and $\beta_{t+1}=\beta_t$.\newline Otherwise, set \mbox{$r^{t+1}=r^t+\bigl(2\beta_t(\hat x-x^t),0\bigr)$}, $x^{t+1}=x^t$, and $\beta_{t+1}=2\beta_t$.
\ENDFOR
\end{algorithmic}
\end{algorithm}

The following theorem establishes the evaluation complexity with the radius $D=2D_{\mathbf y}$, whose proof is deferred to Subsection \ref{subsec:proof-algorithm4}.
\begin{theorem}\label{thm:ncc}
Suppose Assumption~\ref{ass:ncc} holds, and $\mcA$ satisfies Assumption~\ref{ass:subsolver} and the additional execution requirements in Subsection~\ref{sec:pf-scsc}. Let $0<\epsilon<32L_{\nabla h}$. Then Algorithm~\ref{alg:ncc} with $D=2D_{\mathbf y}$ is well-defined and terminates with an $\epsilon$-stationary point $(x,y)$ of~\eqref{eq:ncc}. The numbers of evaluations of $h$, $\nabla h$, $\prox_p$, and $\prox_q$, including all unsuccessful trials, are bounded, respectively, by
\begin{equation}\label{eq:ncc-cost}
\begin{aligned}[b]
&2+2^{13}(1+\sqrt2)\widetilde{\mcC}_{\mcA}\sqrt{\frac{L_{\nabla h}D_{\mathbf y}}\epsilon}
\left(1+\frac{L_{\nabla h}\Delta_0}{\epsilon^2}+\frac{L_{\nabla h}D_{\mathbf y}}\epsilon\right)\\
&=O\!\left(1+\sqrt{\frac{L_{\nabla h}D_{\mathbf y}}{\epsilon}}\left(1+\frac{L_{\nabla h}\Delta_0}{\epsilon^2}+\frac{L_{\nabla h}D_{\mathbf y}}{\epsilon}\right)\right),
\end{aligned}
\end{equation}
where $\Delta_0=\max_y H(x^0,y)-H^*$, and $D_{\mathbf y}$ and $L_{\nabla h}$ are given in Assumption~\ref{ass:ncc}(ii) and (iii), respectively.
\end{theorem}

\begin{remark}
The first primal center is $x^0$, so $\Delta_0$ is the original objective gap. The pair $(\hat x,\hat y)$ is used only as the initial point for the inner calls and is updated after each accepted trial. Algorithm~\ref{alg:ncc} uses the radius bound $D_{\mathbf y}$, but does not require $L_{\nabla h}$ or $\Delta_0$. Its radius test cannot return failure when $D=2D_{\mathbf y}$.
\end{remark}

\subsection{A parameter-free proximal point method for SCSC problems}\label{sec:pf-ncc}

We now remove the need for $D_{\mathbf y}$ by calling Algorithm~\ref{alg:ncc} with successively doubled radius estimates. Each call starts from the original point $(x^0,y^0)$ and initializes its own smoothness estimate. A phase is one complete call to Algorithm~\ref{alg:ncc}: its supplied radius remains fixed, and failure starts the next phase with twice that radius.

\begin{algorithm}[H]
\caption{A parameter-free proximal point method for nonconvex--concave minimax optimization}
\label{alg:pf-ncc}
\begin{algorithmic}[1]
\REQUIRE Subroutine $\mcA$ satisfying Assumption~\ref{ass:subsolver}, $(x^0,y^0)\in\mathcal X\times\mathcal Y$, and $\epsilon>0$.
\FOR{$k=1,2,\ldots$}
\STATE Set $D_k=2^{k-1}$.
\STATE Call Algorithm~\ref{alg:ncc} with $(D,x^0,y^0,\epsilon,\mcA)\leftarrow(D_k,x^0,y^0,\epsilon,\mcA)$ for solving problem~\eqref{eq:ncc}.
\STATE If the call returns a point $(x,y)$, terminate and output $(x,y)$.
\ENDFOR
\end{algorithmic}
\end{algorithm}

The following theorem presents our results on the above algorithm, whose proof is deferred to Subsection \ref{subsec:proof-algorithm4}.
\begin{theorem}\label{thm:pf-ncc}
Suppose Assumption~\ref{ass:ncc} holds, and $\mcA$ satisfies Assumption~\ref{ass:subsolver} and the additional execution requirements in Subsection~\ref{sec:pf-scsc}. Let $0<\epsilon<32L_{\nabla h}$. Then Algorithm~\ref{alg:pf-ncc} is well-defined and terminates with an $\epsilon$-stationary point $(x,y)$ of~\eqref{eq:ncc}. The numbers of evaluations of $h$, $\nabla h$, $\prox_p$, and $\prox_q$, including all calls to Algorithm~\ref{alg:ncc}, are bounded, respectively, by
\begin{equation}\label{eq:pf-ncc-cost}
O\!\left(1+\sqrt{\frac{L_{\nabla h}D_{\mathbf y}}{\epsilon}}\left(1+\frac{L_{\nabla h}\Delta_0}{\epsilon^2}+\frac{L_{\nabla h}D_{\mathbf y}}{\epsilon}\right)\right),
\end{equation}
where $\Delta_0$ is defined in Theorem~\ref{thm:ncc}, and $D_{\mathbf y}$ and $L_{\nabla h}$ are given in Assumption~\ref{ass:ncc}(ii) and (iii), respectively.
\end{theorem}

\begin{remark}
It can be observer from the above theorem that ALgorithm \ref{alg:pf-ncc} enjoys a $O(\epsilon^{-5/2})$ operation complexity for finding an $\epsilon$-stationary point $(x,y)$ of~\eqref{eq:ncc}. To the best of our knowledge, this improves the best known complexity be a logarithmic factor for the first time. Moreover, our method is parameter-free and equipped with a verifiable termination criterion.
\end{remark}

\section{Proof of the main results}
In this section we provide a proof of our main results presented in Sections \ref{sec:algorithm1} and \ref{sec:algorithms}, which are particularly Theorems \ref{thm:contract}--\ref{thm:pf-ncc}.
\subsection{Proof of the main result in Subsection~\ref{sec:contract}}
\label{subsec:proof-algorithm1}

In this subsection we first establish several lemmas and then prove Theorem~\ref{thm:contract}. All iteration indices in this subsection refer to iterations reached by Algorithm~\ref{alg:contract}. Throughout this subsection, $u^*$ denotes the saddle point of~\eqref{eq:scsc}, and $u_*^k$ denotes the solution of the balanced subproblem at $c^k$.

The following lemma establishes that the call to the subsolver terminates in finitely many evaluations.

\begin{lemma}\label{lem:balanced-cost}
Suppose Assumptions~\ref{ass:scsc} and~\ref{ass:subsolver} hold. For every iteration $k$ reached by Algorithm~\ref{alg:contract}, its subsolver initial point $\widehat c^k$ belongs to $\mathcal X\times\mathcal Y$, and its call in step~2 terminates and outputs a pair satisfying the termination criteria~\eqref{eq:relative} within
\[
1+N_{\mcA}+\left\lceil\frac{128\mathcal C_{\mcA}(L_{\nabla f}+\mu_x)}{\mu_x}\right\rceil
\]
evaluations of $\nabla f$, $\prox_p$, and $\prox_q$, respectively.
\end{lemma}
\begin{proof}
Fix any reached iteration $k$. The definition of $\widehat c^k$ in step~2 implies that $\widehat c^k\in\mathcal X\times\mathcal Y$. Thus $\widehat c^k$ is an valid initial point for $\mcA$. It follows from Assumption \ref{ass:scsc} that the balanced subproblem \eqref{eq:bld} has a unique saddle point $u_*^k$. By~\eqref{eq:F-k}, $F_k$ is $(L_{\nabla f}+\mu_x)$-Lipschitz and $\mu_x/2$-strongly monotone on $\mathcal X\times\mathcal Y$, so the parameters in step~2 are valid. Since $0\in(F_k+B)(u_*^k)$, we have $u_*^k=(I+\gamma_kB)^{-1}(u_*^k-\gamma_kF_k(u_*^k))$. By the definition of $\widehat c^k$ in step~2 and the nonexpansivity of the resolvent, one has
\begin{align*}
&\norm{\widehat c^k-u_*^k}\le\norm{c^k-u_*^k-\gamma_k(F_k(u^0)-F_k(u_*^k))}\\
&\le\norm{c^k-u_*^k}+\gamma_k(L_{\nabla f}+\mu_x)\norm{u^0-u_*^k}\\
&=\norm{c^k-u_*^k}+\frac{\norm{u^0-u_*^k}}{(k+2)^2}\le\left(1+\frac1{(k+2)^2}\right)\norm{c^k-u_*^k}+\frac{\norm{c^k-u^0}}{(k+2)^2},
\end{align*}
where the first inequality follows from the nonexpansivity of the resolvent, the second inequality follows from the Lipschitz continuity of $F_k$, and the equality follows from $\gamma_k(L_{\nabla f}+\mu_x)=(k+2)^{-2}$. Unless the call has already terminated, Assumption~\ref{ass:subsolver} with $\vartheta=\mu_x/128$ implies that the current call of $\mcA$ generates a pair $(u,r)$ satisfying $r\in(F_k+B)(u)$ and
\begin{equation}\label{eq:balanced-residual-bound}
\norm r\le\frac{\mu_x}{128}\norm{\widehat c^k-u_*^k}\le\frac{5\mu_x}{512}\norm{c^k-u_*^k}+\frac{\mu_x}{128(k+2)^2}\norm{c^k-u^0},
\end{equation}
where the last inequality follows from $(k+2)^{-2}\le1/4$. This pair is generated in at most $N_{\mcA}+\lceil128\mathcal C_{\mcA}(L_{\nabla f}+\mu_x)/\mu_x\rceil$ evaluations of $F_k$ and of the resolvent of $B$, respectively. By strong monotonicity and $0\in(F_k+B)(u_*^k)$, one has
\[
\frac{\mu_x}{2}\norm{u-u_*^k}^2\le\ip r{u-u_*^k}\le\norm r\norm{u-u_*^k},
\]
and hence $\norm{c^k-u_*^k}\le\norm{u-c^k}+\norm{u-u_*^k}\le\norm{u-c^k}+2\norm r/\mu_x$. Substituting this into~\eqref{eq:balanced-residual-bound}, we obtain
\[
\norm r\le\frac{5\mu_x}{512}\left(\norm{u-c^k}+\frac{2\norm r}{\mu_x}\right)+\frac{\mu_x}{128(k+2)^2}\norm{c^k-u^0}.
\]
Rearranging, we obtain
\begin{align*}
\norm r\le\frac{5\mu_x}{502}\norm{u-c^k}+\frac{2\mu_x}{251(k+2)^2}\norm{c^k-u^0}\le\frac{\mu_x}{64}\left(\norm{u-c^k}+\frac{\norm{c^k-u^0}}{(k+2)^2}\right).
\end{align*}
Hence, this pair satisfies~\eqref{eq:relative}. By~\eqref{eq:F-k},
\[
F_k(u^0)=F(u^0)+\left(-\frac{\mu_x}{2}(x^0-c_x^k),\mu_x(y^0-c_y^k)\right).
\]
Thus the initialization evaluates each of $\prox_p$ and $\prox_q$ once and computes $F_k(u^0)$ from the previously evaluated $F(u^0)$ by affine arithmetic. Each subsequent evaluation of $F_k$ requires one evaluation of $\nabla f$, and each resolvent evaluation requires one evaluation of each proximal mapping. The conclusion then follows from these.
\end{proof}

The next lemma establishes the inclusions and relative-error bound for $e^{k+1}$ and $\tilde u^{k+1}$ in~\eqref{eq:algorithm1-updates} and~\eqref{eq:algorithm1-corrected-point}, respectively.

\begin{lemma}\label{lem:certificate}
Suppose Assumptions~\ref{ass:scsc} and~\ref{ass:subsolver} hold. Let $F$ and $B$ be defined in~\eqref{eq:scsc-operators}, and $\Phi$ in~\eqref{eq:convex-reformulation}. For the sequence generated by Algorithm~\ref{alg:contract} from $u^0$, let $e^{k+1}$, $\tilde u^{k+1}$, and $c^{k+1}$ be defined by~\eqref{eq:algorithm1-updates}--\eqref{eq:algorithm1-center}. Then $e^{k+1}\in(F+B)(u^{k+1})$ and $(-e_x^{k+1},e_y^{k+1})\in\partial\Phi(\tilde u^{k+1})$ at every reached iteration $k$. Moreover, we have
\begin{equation}\label{eq:relative-proximal}
\left\|\mu_x(\tilde u^{k+1}-c^k)+(-e_x^{k+1},e_y^{k+1})\right\|\le\frac{\mu_x}{8}\left(\norm{\tilde u^{k+1}-c^k}+\frac{\norm{c^k-u^0}}{(k+2)^2}\right).
\end{equation}
\end{lemma}
\begin{proof}
Fix any reached iteration $k$. By step~2 of Algorithm~\ref{alg:contract} and~\eqref{eq:algorithm1-updates}, $r^{k+1}-F_k(u^{k+1})\in B(u^{k+1})$ and
\[
e^{k+1}=F(u^{k+1})+\bigl(r^{k+1}-F_k(u^{k+1})\bigr)\in(F+B)(u^{k+1}).
\]
In particular, $e_x^{k+1}\in\nabla_x f(u^{k+1})+\partial p(x^{k+1})$. In addition, it follows from~\eqref{eq:algorithm1-corrected-point} that $e_x^{k+1}=\mu_x(x^{k+1}-\tilde u_x^{k+1})$. Using this and strong convexity, we obtain that for every $x\in\mathcal X$,
\begin{equation}\label{eq:primal-certificate}
\begin{aligned}[b]
&f(x,y^{k+1})+p(x)\ge f(u^{k+1})+p(x^{k+1})+\ip{e_x^{k+1}}{x-x^{k+1}}+\frac{\mu_x}{2}\norm{x-x^{k+1}}^2\\
&=f(u^{k+1})+p(x^{k+1})+\ip{e_x^{k+1}}{x-\tilde u_x^{k+1}-e_x^{k+1}/\mu_x}+\frac{\mu_x}{2}\norm{x-\tilde u_x^{k+1}-e_x^{k+1}/\mu_x}^2\\
&=f(u^{k+1})+p(x^{k+1})+\frac{\mu_x}{2}\norm{x-\tilde u_x^{k+1}}^2-\frac{\norm{e_x^{k+1}}^2}{2\mu_x},
\end{aligned}
\end{equation}
where the first equality follows from $e_x^{k+1}=\mu_x(x^{k+1}-\tilde u_x^{k+1})$. Notice that equality holds in the first inequality at $x=x^{k+1}$. Thus the supremum in $\Phi(\tilde u^{k+1})$ in~\eqref{eq:convex-reformulation} is attained at $x^{k+1}$. Since $\tilde u_y^{k+1}=y^{k+1}$,
\begin{equation}\label{eq:phi-certificate-value}
\begin{aligned}[b]
&\Phi(\tilde u^{k+1})=q(y^{k+1})+\sup_x\left\{\frac{\mu_x}{2}\norm{x-\tilde u_x^{k+1}}^2-f(x,y^{k+1})-p(x)\right\}\\
&=q(y^{k+1})-f(u^{k+1})-p(x^{k+1})+\frac{\mu_x}{2}\norm{x^{k+1}-\tilde u_x^{k+1}}^2\\
&=q(y^{k+1})-f(u^{k+1})-p(x^{k+1})+\frac{\norm{e_x^{k+1}}^2}{2\mu_x},
\end{aligned}
\end{equation}
where the second equality follows from~\eqref{eq:primal-certificate}, and the last uses $x^{k+1}-\tilde u_x^{k+1}=e_x^{k+1}/\mu_x$. For any $(\xi,y)\in\dom\Phi$, it follows from $x^{k+1}=\tilde u_x^{k+1}+e_x^{k+1}/\mu_x$ that
\begin{equation}\label{eq:shifted-primal-square}
\frac{\mu_x}{2}\norm{x^{k+1}-\xi}^2=\frac{\norm{e_x^{k+1}}^2}{2\mu_x}-\ip{e_x^{k+1}}{\xi-\tilde u_x^{k+1}}+\frac{\mu_x}{2}\norm{\xi-\tilde u_x^{k+1}}^2.
\end{equation}
Using this,~\eqref{eq:convex-reformulation}, and~\eqref{eq:phi-certificate-value}, we obtain
\begin{align*}
&\Phi(\xi,y)\ge q(y)-f(x^{k+1},y)-p(x^{k+1})+\frac{\mu_x}{2}\norm{x^{k+1}-\xi}^2\\
&\ge q(y^{k+1})-f(u^{k+1})-p(x^{k+1})+\ip{e_y^{k+1}}{y-y^{k+1}}\\
&\quad+\frac{\mu_y}{2}\norm{y-y^{k+1}}^2+\frac{\mu_x}{2}\norm{x^{k+1}-\xi}^2\\
&=\Phi(\tilde u^{k+1})+\ip{(-e_x^{k+1},e_y^{k+1})}{(\xi,y)-\tilde u^{k+1}}+\frac{\mu_x}{2}\norm{\xi-\tilde u_x^{k+1}}^2+\frac{\mu_y}{2}\norm{y-y^{k+1}}^2,
\end{align*}
where the first inequality follows from the definition of the supremum and the choice $x^{k+1}$, the second inequality is due to the $\mu_y$-strong convexity of $q(\cdot)-f(x^{k+1},\cdot)$ with subgradient $e_y^{k+1}$ at $y^{k+1}$, and the equality follows from~\eqref{eq:phi-certificate-value} and~\eqref{eq:shifted-primal-square}. Hence, $(-e_x^{k+1},e_y^{k+1})\in\partial\Phi(\tilde u^{k+1})$.

We next show that~\eqref{eq:relative-proximal} holds. By~\eqref{eq:F-k} and~\eqref{eq:algorithm1-updates}, one has
\[
e^{k+1}=r^{k+1}+\left(\frac{\mu_x}{2}(x^{k+1}-c_x^k),-\mu_x(y^{k+1}-c_y^k)\right).
\]
Using this and~\eqref{eq:algorithm1-corrected-point}, we obtain
\begin{equation}\label{eq:corrected-point-expansion}
\tilde u^{k+1}=\left(x^{k+1}-\frac{e_x^{k+1}}{\mu_x},y^{k+1}\right)=\left(x^{k+1}-\frac{r_x^{k+1}}{\mu_x}-\frac{x^{k+1}-c_x^k}{2},y^{k+1}\right)=\left(\frac{x^{k+1}+c_x^k}{2}-\frac{r_x^{k+1}}{\mu_x},y^{k+1}\right),
\end{equation}
where the second equality follows from $e_x^{k+1}=r_x^{k+1}+\mu_x(x^{k+1}-c_x^k)/2$. In addition, we have
\begin{equation}\label{eq:point-residual-identities}
\begin{aligned}[b]
&\mu_x(\tilde u^{k+1}-c^k)+(-e_x^{k+1},e_y^{k+1})=\bigl(\mu_x(x^{k+1}-c_x^k)-2e_x^{k+1},\mu_x(y^{k+1}-c_y^k)+e_y^{k+1}\bigr)\\
&=\left(\mu_x(x^{k+1}-c_x^k)-2\left(r_x^{k+1}+\frac{\mu_x}{2}(x^{k+1}-c_x^k)\right),r_y^{k+1}\right)=(-2r_x^{k+1},r_y^{k+1}),
\end{aligned}
\end{equation}
where the first equality follows from $\tilde u_x^{k+1}=x^{k+1}-e_x^{k+1}/\mu_x$ and $\tilde u_y^{k+1}=y^{k+1}$, and the second follows from~\eqref{eq:F-k} and~\eqref{eq:algorithm1-updates}.

By~\eqref{eq:corrected-point-expansion}, $x^{k+1}-c_x^k=2(\tilde u_x^{k+1}-c_x^k)+2r_x^{k+1}/\mu_x$ and $y^{k+1}-c_y^k=\tilde u_y^{k+1}-c_y^k$. Using these, we obtain
\begin{align*}
&\norm{u^{k+1}-c^k}=\left\|\bigl(2(\tilde u_x^{k+1}-c_x^k),\tilde u_y^{k+1}-c_y^k\bigr)+\left(\frac{2r_x^{k+1}}{\mu_x},0\right)\right\|\\
&\le\sqrt{4\norm{\tilde u_x^{k+1}-c_x^k}^2+\norm{\tilde u_y^{k+1}-c_y^k}^2}+\frac{2}{\mu_x}\norm{r_x^{k+1}}\le2\norm{\tilde u^{k+1}-c^k}+\frac{2}{\mu_x}\norm{r^{k+1}}.
\end{align*}
Substituting this bound into~\eqref{eq:relative} yields
\begin{align*}
\norm{r^{k+1}}\le\frac{\mu_x}{64}\left(\norm{u^{k+1}-c^k}+\frac{\norm{c^k-u^0}}{(k+2)^2}\right)\le\frac{\mu_x}{32}\norm{\tilde u^{k+1}-c^k}+\frac1{32}\norm{r^{k+1}}+\frac{\mu_x}{64(k+2)^2}\norm{c^k-u^0},
\end{align*}
which further implies that
\begin{equation}\label{eq:inner-residual-bound}
\norm{r^{k+1}}\le\frac{\mu_x}{31}\norm{\tilde u^{k+1}-c^k}+\frac{\mu_x}{62(k+2)^2}\norm{c^k-u^0}.
\end{equation}
Using this and~\eqref{eq:point-residual-identities}, we have
\begin{align*}
&\left\|\mu_x(\tilde u^{k+1}-c^k)+(-e_x^{k+1},e_y^{k+1})\right\|=\sqrt{4\norm{r_x^{k+1}}^2+\norm{r_y^{k+1}}^2}\le2\norm{r^{k+1}}\\
&\le\frac{2\mu_x}{31}\norm{\tilde u^{k+1}-c^k}+\frac{\mu_x}{31(k+2)^2}\norm{c^k-u^0}\le\frac{\mu_x}{8}\left(\norm{\tilde u^{k+1}-c^k}+\frac{\norm{c^k-u^0}}{(k+2)^2}\right),
\end{align*}
where the second inequality follows from~\eqref{eq:inner-residual-bound}. Hence,~\eqref{eq:relative-proximal} holds.
\end{proof}

\Needspace{12\baselineskip}
The next three lemmas establish an energy estimate, a uniform bound, and the point and residual bounds for Algorithm~\ref{alg:contract}.

\begin{lemma}
Suppose Assumptions~\ref{ass:scsc} and~\ref{ass:subsolver} hold. Let $u^*$ be the unique saddle point of~\eqref{eq:scsc}, and let $\Phi$ be defined in~\eqref{eq:convex-reformulation}. For the sequence generated by Algorithm~\ref{alg:contract}, define
\begin{equation}\label{eq:auxiliary-point}
\bar u^i=\frac{i+2}{2}c^i-\frac{i}{2}\tilde u^i
\end{equation}
at each generated index $i$. Then, at every reached iteration $k$,
\begin{equation}\label{eq:additive-energy}
\begin{aligned}[b]
&\frac{(k+2)^2}{4\mu_x}\bigl(\Phi(\tilde u^{k+1})-\Phi(u^*)\bigr)+\frac12\norm{\bar u^{k+1}-u^*}^2+\frac{31}{256}\sum_{i=0}^k(i+2)^2\norm{\tilde u^{i+1}-c^i}^2\\
&\le\frac12\norm{u^0-u^*}^2+\frac1{256}\sum_{i=0}^k\frac{\norm{c^i-u^0}^2}{(i+2)^2}.
\end{aligned}
\end{equation}
\end{lemma}
\begin{proof}
Fix any reached iteration $k$. By Lemma~\ref{lem:certificate}, $(-e_x^{i+1},e_y^{i+1})\in\partial\Phi(\tilde u^{i+1})$ for $0\le i\le k$. By the subgradient inequality, one has
\begin{equation}\label{eq:successive-comparison}
\Phi(\tilde u^{i+1})-\Phi(\tilde u^i)\le\ip{(-e_x^{i+1},e_y^{i+1})}{\tilde u^{i+1}-\tilde u^i}\qquad \forall1\le i\le k.
\end{equation}
In addition, by the $\mu_y$-strong convexity of $\Phi$ and $0\in\partial\Phi(u^*)$, we have
\begin{equation}\label{eq:gap-bounds}
\frac{\mu_y}{2}\norm{\tilde u^{i+1}-u^*}^2\le\Phi(\tilde u^{i+1})-\Phi(u^*)\le\ip{(-e_x^{i+1},e_y^{i+1})}{\tilde u^{i+1}-u^*}\qquad\forall0\le i\le k.
\end{equation}
By~\eqref{eq:auxiliary-point}, $\bar u^0=u^0$. Using this definition and~\eqref{eq:algorithm1-center}, we obtain, for $0\le i\le k$,
\begin{align*}
&\bar u^{i+1}=\frac{i+3}{2}\left(\tilde u^{i+1}+\frac{i}{i+3}(\tilde u^{i+1}-\tilde u^i)+\frac{i+2}{(i+3)\mu_x}(2r_x^{i+1},-r_y^{i+1})\right)-\frac{i+1}{2}\tilde u^{i+1}\\
&=\frac{i+2}{2}\tilde u^{i+1}-\frac{i}{2}\tilde u^i+\frac{i+2}{2\mu_x}(2r_x^{i+1},-r_y^{i+1}).
\end{align*}
Recall from~\eqref{eq:point-residual-identities} that $(2r_x^{i+1},-r_y^{i+1})=-\mu_x(\tilde u^{i+1}-c^i)-(-e_x^{i+1},e_y^{i+1})$. Substituting this identity into the displayed expression for $\bar u^{i+1}$ and using~\eqref{eq:auxiliary-point}, we obtain
\[
\bar u^{i+1}=\frac{i+2}{2}c^i-\frac{i}{2}\tilde u^i-\frac{i+2}{2\mu_x}(-e_x^{i+1},e_y^{i+1})=\bar u^i-\frac{i+2}{2\mu_x}(-e_x^{i+1},e_y^{i+1}).
\]
It then follows that
\begin{equation}\label{eq:auxiliary-norm-difference}
\frac12\norm{\bar u^{i+1}-u^*}^2-\frac12\norm{\bar u^i-u^*}^2=-\frac{i+2}{2\mu_x}\ip{(-e_x^{i+1},e_y^{i+1})}{\bar u^i-u^*}+\frac{(i+2)^2}{8\mu_x^2}\norm{e^{i+1}}^2.
\end{equation}
For $1\le i\le k$, multiplying~\eqref{eq:successive-comparison} by $i(i+2)/(4\mu_x)$ and the upper bound in~\eqref{eq:gap-bounds} by $(i+2)/(2\mu_x)$ and adding them, we obtain
\begin{equation}\label{eq:weighted-gap}
\begin{aligned}[b]
&\frac{(i+2)^2}{4\mu_x}\bigl(\Phi(\tilde u^{i+1})-\Phi(u^*)\bigr)-\frac{i(i+2)}{4\mu_x}\bigl(\Phi(\tilde u^i)-\Phi(u^*)\bigr)\\
&=\frac{i(i+2)}{4\mu_x}\bigl(\Phi(\tilde u^{i+1})-\Phi(\tilde u^i)\bigr)+\frac{i+2}{2\mu_x}\bigl(\Phi(\tilde u^{i+1})-\Phi(u^*)\bigr)\\
&\le\frac{i(i+2)}{4\mu_x}\ip{(-e_x^{i+1},e_y^{i+1})}{\tilde u^{i+1}-\tilde u^i}+\frac{i+2}{2\mu_x}\ip{(-e_x^{i+1},e_y^{i+1})}{\tilde u^{i+1}-u^*}\\
&=\frac{i+2}{4\mu_x}\ip{(-e_x^{i+1},e_y^{i+1})}{(i+2)\tilde u^{i+1}-i\tilde u^i-2u^*}\\
&=\frac{(i+2)^2}{4\mu_x}\ip{(-e_x^{i+1},e_y^{i+1})}{\tilde u^{i+1}-c^i}+\frac{i+2}{2\mu_x}\ip{(-e_x^{i+1},e_y^{i+1})}{\bar u^i-u^*},
\end{aligned}
\end{equation}
where the inequality follows from~\eqref{eq:successive-comparison} and~\eqref{eq:gap-bounds} and the last equality follows from~\eqref{eq:auxiliary-point}.
Using~\eqref{eq:relative-proximal} to bound the sum of~\eqref{eq:auxiliary-norm-difference} and~\eqref{eq:weighted-gap}, we obtain
\begin{equation}\label{eq:one-step}
\begin{aligned}[b]
&\frac{(i+2)^2}{4\mu_x}\bigl(\Phi(\tilde u^{i+1})-\Phi(u^*)\bigr)+\frac12\norm{\bar u^{i+1}-u^*}^2-\frac{(i+1)^2}{4\mu_x}\bigl(\Phi(\tilde u^i)-\Phi(u^*)\bigr)-\frac12\norm{\bar u^i-u^*}^2\\
&\le\frac{(i+2)^2}{4\mu_x}\ip{(-e_x^{i+1},e_y^{i+1})}{\tilde u^{i+1}-c^i}+\frac{(i+2)^2}{8\mu_x^2}\norm{e^{i+1}}^2\\
&=\frac{(i+2)^2}{8\mu_x^2}\left(\left\|\mu_x(\tilde u^{i+1}-c^i)+(-e_x^{i+1},e_y^{i+1})\right\|^2-\mu_x^2\norm{\tilde u^{i+1}-c^i}^2\right)\\
&\le\frac{(i+2)^2}{512}\left(\norm{\tilde u^{i+1}-c^i}+\frac{\norm{c^i-u^0}}{(i+2)^2}\right)^2-\frac{(i+2)^2}{8}\norm{\tilde u^{i+1}-c^i}^2\\
&\le-\frac{31(i+2)^2}{256}\norm{\tilde u^{i+1}-c^i}^2+\frac{\norm{c^i-u^0}^2}{256(i+2)^2},
\end{aligned}
\end{equation}
where the first inequality also uses $i(i+2)\le(i+1)^2$ and $\Phi(\tilde u^i)\ge\Phi(u^*)$, the second inequality follows from~\eqref{eq:relative-proximal}, and the last inequality is due to the convexity of $|\cdot|^2$.

At $i=0$, multiplying the upper bound in~\eqref{eq:gap-bounds} by $1/\mu_x$ and adding~\eqref{eq:auxiliary-norm-difference}, we obtain
\begin{align*}
&\frac{1}{\mu_x}\left(\Phi(\tilde u^1)-\Phi(u^*)\right)+\frac12\norm{\bar u^1-u^*}^2-\frac12\norm{u^0-u^*}^2\le\frac1{\mu_x}\ip{(-e_x^1,e_y^1)}{\tilde u^1-c^0}+\frac1{2\mu_x^2}\norm{e^1}^2\\
&=\frac1{2\mu_x^2}\left(\left\|\mu_x(\tilde u^1-c^0)+(-e_x^1,e_y^1)\right\|^2-\mu_x^2\norm{\tilde u^1-c^0}^2\right)\le\frac12\left(\frac1{64}-1\right)\norm{\tilde u^1-c^0}^2\le-\frac{31}{64}\norm{\tilde u^1-c^0}^2,
\end{align*}
where the first inequality follows from $\bar u^0=c^0=u^0$ and \eqref{eq:auxiliary-norm-difference}, and the second inequality follows from~\eqref{eq:relative-proximal} and $c^0=u^0$. Summing this inequality and~\eqref{eq:one-step} for $i=1,\ldots,k$ yields~\eqref{eq:additive-energy}.
\end{proof}

The following lemma bounds the accumulated error in~\eqref{eq:additive-energy}.

\begin{lemma}
Suppose Assumptions~\ref{ass:scsc} and~\ref{ass:subsolver} hold. Let $u^*$ be the unique saddle point of~\eqref{eq:scsc}, $\Phi$ be defined in~\eqref{eq:convex-reformulation}, and $\bar u^i$ be defined by~\eqref{eq:auxiliary-point}. Then the sequence generated by Algorithm~\ref{alg:contract} satisfies $\norm{c^k-u^0}\le3\norm{u^0-u^*}$ and
\begin{equation}\label{eq:potential}
\frac{(k+2)^2}{4\mu_x}\bigl(\Phi(\tilde u^{k+1})-\Phi(u^*)\bigr)+\frac12\norm{\bar u^{k+1}-u^*}^2+\frac1{16}\sum_{i=0}^k(i+2)^2\norm{\tilde u^{i+1}-c^i}^2\le\norm{u^0-u^*}^2
\end{equation}
at every reached iteration $k$.
\end{lemma}
\begin{proof}
Fix any reached iteration $k$. We show by induction that
\begin{equation}\label{eq:uniform-iterate-bounds}
\norm{\tilde u^i-u^*}\le2\norm{u^0-u^*},\qquad\norm{\bar u^i-u^*}\le\sqrt2\norm{u^0-u^*}\qquad\forall0\le i\le k+1.
\end{equation}
Both bounds hold at $i=0$ by definition. Suppose that they hold through some $t\le k$. By~\eqref{eq:auxiliary-point}, one has for $0\le i\le t$ that
\begin{equation}\label{eq:inductive-center-distance}
\begin{aligned}[b]
&\norm{c^i-u^0}=\left\|\frac{i}{i+2}(\tilde u^i-u^*)+\frac2{i+2}(\bar u^i-u^*)+(u^*-u^0)\right\|\\
&\le\left(\frac{2i+2\sqrt2}{i+2}+1\right)\norm{u^0-u^*}\le3\norm{u^0-u^*},
\end{aligned}
\end{equation}
where the first inequality follows from the induction hypothesis. Using this, we obtain
\begin{align*}
&\frac12\norm{u^0-u^*}^2+\frac1{256}\sum_{i=0}^t\frac{\norm{c^i-u^0}^2}{(i+2)^2}\le\left(\frac12+\frac9{256}\sum_{i=0}^t\frac1{(i+2)^2}\right)\norm{u^0-u^*}^2\le\norm{u^0-u^*}^2,
\end{align*}
where the last inequality follows from $\sum_{i=0}^t(i+2)^{-2}\le1$. By this,~\eqref{eq:additive-energy}, and $31/256\ge1/16$, the bound~\eqref{eq:potential} holds with $k=t$. Since all terms on the left are nonnegative, one has
\begin{equation}\label{eq:inductive-output-bounds}
\norm{\bar u^{t+1}-u^*}\le\sqrt2\norm{u^0-u^*},\qquad\norm{\tilde u^{t+1}-c^t}\le\frac4{t+2}\norm{u^0-u^*}.
\end{equation}
In addition, by~\eqref{eq:algorithm1-center} and~\eqref{eq:auxiliary-point}, we obtain
\begin{align*}
&\norm{\tilde u^{t+1}-u^*}=\left\|\frac{t}{t+2}(\tilde u^t-u^*)+\frac2{t+2}(\bar u^{t+1}-u^*)+\frac1{\mu_x}(-2r_x^{t+1},r_y^{t+1})\right\|\\
&\le\frac{t}{t+2}\norm{\tilde u^t-u^*}+\frac2{t+2}\norm{\bar u^{t+1}-u^*}+\frac18\left(\norm{\tilde u^{t+1}-c^t}+\frac{\norm{c^t-u^0}}{(t+2)^2}\right)\\
&\le\frac{2t+2\sqrt2+3/4}{t+2}\norm{u^0-u^*}\le2\norm{u^0-u^*},
\end{align*}
where the first inequality follows from~\eqref{eq:relative-proximal} and~\eqref{eq:point-residual-identities}, and the second follows from~\eqref{eq:inductive-center-distance},~\eqref{eq:inductive-output-bounds}, the induction hypothesis, and $t+2\ge2$. This completes the induction and establishes~\eqref{eq:potential}. Finally, by~\eqref{eq:auxiliary-point} and~\eqref{eq:uniform-iterate-bounds}, one has
\begin{equation}\label{eq:center-distance}
\norm{c^k-u^0}\le\frac{k}{k+2}\norm{\tilde u^k-u^*}+\frac2{k+2}\norm{\bar u^k-u^*}+\norm{u^0-u^*}\le3\norm{u^0-u^*}.
\end{equation}
Hence, the conclusion follows.
\end{proof}

We can now establish the point and residual bounds.

\begin{lemma}
Suppose Assumptions~\ref{ass:scsc} and~\ref{ass:subsolver} hold. Let $u^*$ be the unique saddle point of~\eqref{eq:scsc}, and let $e^{k+1}$ and $\tilde u^{k+1}$ be defined by~\eqref{eq:algorithm1-updates} and~\eqref{eq:algorithm1-corrected-point}, respectively, in Algorithm~\ref{alg:contract}. Then, at every reached iteration $k$, we have
\begin{equation}\label{eq:point-rate}
\norm{\tilde u^{k+1}-u^*}\le\frac{3}{k+2}\sqrt{\frac{\mu_x}{\mu_y}}\norm{u^0-u^*},\qquad \norm{e^{k+1}}\le\frac{5\mu_x}{k+2}\norm{u^0-u^*}.
\end{equation}
\end{lemma}
\begin{proof}
Fix any reached iteration $k$. By the lower bound in~\eqref{eq:gap-bounds} and~\eqref{eq:potential}, one has
\[
\frac{(k+2)^2\mu_y}{8\mu_x}\norm{\tilde u^{k+1}-u^*}^2\le\frac{(k+2)^2}{4\mu_x}\bigl(\Phi(\tilde u^{k+1})-\Phi(u^*)\bigr)\le\norm{u^0-u^*}^2.
\]
The first bound in~\eqref{eq:point-rate} follows from this and $\sqrt8\le3$. In addition, we have
\begin{align*}
&\norm{e^{k+1}}\le\left\|\mu_x(\tilde u^{k+1}-c^k)+(-e_x^{k+1},e_y^{k+1})\right\|+\mu_x\norm{\tilde u^{k+1}-c^k}\\
&\le\frac{9\mu_x}{8}\norm{\tilde u^{k+1}-c^k}+\frac{\mu_x}{8(k+2)^2}\norm{c^k-u^0}\\
&\le\left(\frac9{2(k+2)}+\frac3{8(k+2)^2}\right)\mu_x\norm{u^0-u^*}\le\frac{5\mu_x}{k+2}\norm{u^0-u^*},
\end{align*}
where the first inequality follows from the triangle inequality, the second inequality follows from~\eqref{eq:relative-proximal}, and the third inequality follows from~\eqref{eq:potential} and~\eqref{eq:center-distance}. Hence, the second bound in~\eqref{eq:point-rate} holds.
\end{proof}

We are now ready to prove Theorem~\ref{thm:contract}.

\begin{proof}[Proof of Theorem~\ref{thm:contract}]
By Lemma~\ref{lem:balanced-cost}, each call in step~2 starts at a point in $\mathcal X\times\mathcal Y$ and terminates with a pair satisfying~\eqref{eq:relative}. Since the remaining updates are explicit, Algorithm~\ref{alg:contract} is well-defined at every reached iteration. In addition, it follows from Lemma~\ref{lem:certificate} that $e^{k+1}\in(F+B)(u^{k+1})$.

Suppose that Algorithm~\ref{alg:contract} terminates at iteration $k$. If $e^{k+1}=0$, then the strong monotonicity of $F+B$ yields $u^{k+1}=u^*$, so~\eqref{eq:contract-guarantee} holds. Otherwise, by~\eqref{eq:contract-stop}, one has $1/9-3\sqrt{\mu_x/\mu_y}/(k+2)>0$. Using this,~\eqref{eq:algorithm1-corrected-point}, and~\eqref{eq:point-rate}, we obtain
\begin{align*}
&\norm{u^{k+1}-u^*}=\left\|\tilde u^{k+1}-u^*+\frac1{\mu_x}(e_x^{k+1},0)\right\|\le\norm{\tilde u^{k+1}-u^*}+\frac{\norm{e^{k+1}}}{\mu_x}\\
&\le\frac3{k+2}\sqrt{\frac{\mu_x}{\mu_y}}\norm{u^0-u^*}+\left(\frac19-\frac3{k+2}\sqrt{\frac{\mu_x}{\mu_y}}\right)\norm{u^{k+1}-u^0}\\
&\le\frac3{k+2}\sqrt{\frac{\mu_x}{\mu_y}}\norm{u^0-u^*}+\left(\frac19-\frac3{k+2}\sqrt{\frac{\mu_x}{\mu_y}}\right)\bigl(\norm{u^{k+1}-u^*}+\norm{u^0-u^*}\bigr)\\
&=\frac19\norm{u^0-u^*}+\left(\frac19-\frac3{k+2}\sqrt{\frac{\mu_x}{\mu_y}}\right)\norm{u^{k+1}-u^*},
\end{align*}
where the second inequality follows from~\eqref{eq:contract-stop} and~\eqref{eq:point-rate}. Rearranging, we obtain $\norm{u^{k+1}-u^*}\le\norm{u^0-u^*}/8$. Using this and~\eqref{eq:contract-stop}, we obtain
\[
\norm{e^{k+1}}\le\frac{\mu_x}{9}\norm{u^{k+1}-u^0}\le\frac{\mu_x}{9}\bigl(\norm{u^{k+1}-u^*}+\norm{u^0-u^*}\bigr)\le\frac{\mu_x}{8}\norm{u^0-u^*}.
\]
Thus~\eqref{eq:contract-guarantee} holds whenever the algorithm terminates.

We next establish the iteration bound. Suppose that it reaches iteration $k=\lceil128\sqrt{\mu_x/\mu_y}\rceil-1$. By~\eqref{eq:algorithm1-corrected-point} and~\eqref{eq:point-rate}, one has
\begin{align*}
&\norm{u^{k+1}-u^*}\le\norm{\tilde u^{k+1}-u^*}+\frac{\norm{e^{k+1}}}{\mu_x}\le\frac{3\sqrt{\mu_x/\mu_y}+5}{k+2}\norm{u^0-u^*}\\
&\le\frac{8\sqrt{\mu_x/\mu_y}}{128\sqrt{\mu_x/\mu_y}}\norm{u^0-u^*}=\frac1{16}\norm{u^0-u^*},
\end{align*}
where the last inequality follows from $k+2\ge128\sqrt{\mu_x/\mu_y}$ and $\mu_x\ge\mu_y$. It then follows that
\[
\norm{u^{k+1}-u^0}\ge\norm{u^0-u^*}-\norm{u^{k+1}-u^*}\ge\frac{15}{16}\norm{u^0-u^*}.
\]
Using this and the second bound in~\eqref{eq:point-rate}, we obtain
\begin{align*}
\norm{e^{k+1}}\le\frac{5\mu_x}{k+2}\norm{u^0-u^*}\le\frac{5\mu_x}{128}\norm{u^0-u^*}\le\frac{\mu_x}{24}\norm{u^{k+1}-u^0}\le\mu_x\left(\frac19-\frac3{k+2}\sqrt{\frac{\mu_x}{\mu_y}}\right)\norm{u^{k+1}-u^0},
\end{align*}
where the last inequality follows from $1/9-3\sqrt{\mu_x/\mu_y}/(k+2)\ge1/9-3/128\ge1/24$. Thus~\eqref{eq:contract-stop} holds, and the algorithm terminates within $\lceil128\sqrt{\mu_x/\mu_y}\rceil$ iterations.

The iteration bound and Lemma~\ref{lem:balanced-cost} yield the explicit evaluation bound in~\eqref{eq:algorithm1-cost}. Its leading term $1$ accounts for the initial evaluation of $\nabla f$ and is not needed for $\prox_p$ or $\prox_q$. Since $L_{\nabla f}/\mu_x\ge1$ and $\mu_x/\mu_y\ge1$, the same bound has the stated asymptotic order.
\end{proof}
\subsection{Proof of the main result in Subsection~\ref{sec:scsc}}
\label{subsec:proof-algorithm2}

In this subsection we first prove two lemmas for Algorithm~\ref{alg:scsc} and then prove Theorem~\ref{thm:scsc}.

\begin{lemma}\label{lem:algorithm2-well-defined}
Suppose that Assumptions~\ref{ass:scsc} and~\ref{ass:subsolver} hold. Let $u^*$ be the unique saddle point of~\eqref{eq:scsc}, and let $F$ and $B$ be defined in~\eqref{eq:scsc-operators}. Then Algorithm~\ref{alg:scsc}, initialized at $u^0$, is well-defined at every reached iteration. The pair $(u^1,e^1)$ generated in step~1 satisfies $e^1\in(F+B)(u^1)$ and
\begin{equation}\label{eq:algorithm2-initial-distance}
7\norm{u^1-u^*}\le\norm{u^1-u^0}\le\frac98\norm{u^0-u^*}.
\end{equation}
At every reached regularized iteration $k\ge1$, the subproblem in step~3 has a unique saddle point $u_*^k$. Moreover, $u^{k+1}$ and $e^{k+1}$, defined in steps~3 and~4, respectively, satisfy $u^{k+1}\in\mathcal X\times\mathcal Y$ and $e^{k+1}\in(F+B)(u^{k+1})$.
\end{lemma}
\begin{proof}
By Theorem~\ref{thm:contract}, the preliminary call terminates with $u^1\in\mathcal X\times\mathcal Y$, $e^1\in(F+B)(u^1)$, and $\norm{u^1-u^*}\le\norm{u^0-u^*}/8$. Using this and the triangle inequality, we obtain
\[
8\norm{u^1-u^*}\le\norm{u^1-u^0}+\norm{u^1-u^*},\qquad
\norm{u^1-u^0}\le\frac98\norm{u^0-u^*}.
\]
It then follows that~\eqref{eq:algorithm2-initial-distance} holds. If $u^1=u^0$, this inequality implies $u^0=u^*$. By this and~\eqref{eq:contract-guarantee}, one has $e^1=0$, and hence the algorithm terminates in step~1. Therefore, whenever the regularized loop is reached, $\norm{u^1-u^0}>0$, and the definitions of $\lambda_1$ and $\bar u^1$ in step~1 are well-defined with $\lambda_1>0$ and $\bar u^1=u^1$.

We next show by induction that $u^k\in\mathcal X\times\mathcal Y$, $\lambda_k>0$, and $\bar u^k$ is well-defined at every reached regularized iteration $k$. These properties hold at $k=1$ by step~1. Suppose that they hold at some reached iteration $k\ge1$. By Assumption~\ref{ass:scsc} and the definition of $f_k$ in step~3, $f_k$ has an $(L_{\nabla f}+\lambda_k)$-Lipschitz gradient and is $(\mu_x+\lambda_k)$-strongly convex in $x$ and $(\mu_y+\lambda_k)$-strongly concave in $y$ on $\mathcal X\times\mathcal Y$. Hence the regularized problem satisfies Assumption~\ref{ass:scsc} with the supplied parameters and has a unique saddle point $u_*^k$. It then follows from Theorem~\ref{thm:contract} that the call in step~3 terminates with $u^{k+1}\in\mathcal X\times\mathcal Y$ and
\[
r^{k+1}\in(F+B)(u^{k+1})+\lambda_k(u^{k+1}-\bar u^k).
\]
Using this and the definition of $e^{k+1}$ in step~4, we obtain $e^{k+1}\in(F+B)(u^{k+1})$. If the algorithm does not terminate in step~5, then $\lambda_{k+1}=4\lambda_k>0$ and $\bar u^{k+1}=(\bar u^k+3u^{k+1})/4$ are well-defined. Thus the induction hypothesis holds at iteration $k+1$. Hence the induction is completed and the conclusion of this lemma holds.
\end{proof}

The next lemma bounds the residual $e^{k+1}$ defined in step~4 of Algorithm~\ref{alg:scsc}.

\begin{lemma}
Suppose that Assumptions~\ref{ass:scsc} and~\ref{ass:subsolver} hold. Let $u^*$ be the unique saddle point of~\eqref{eq:scsc}. For the sequence generated by Algorithm~\ref{alg:scsc}, let $u^1$ and $\lambda_1$ be defined in step~1, $\lambda_k$ by step~5 for $k\ge2$, and $e^{k+1}$ by step~4. Then, at every reached regularized iteration $k\ge1$, we have
\begin{equation}\label{eq:stage-residual}
\norm{e^{k+1}}\le\left(\frac94\lambda_1+(\mu_x+2\lambda_k)8^{-k}\right)\norm{u^1-u^*}.
\end{equation}
\end{lemma}
\begin{proof}
Fix any reached regularized iteration $k\ge1$. Let $u_*^i$ be the unique saddle point of the subproblem in step~3 at iteration $i$, whose existence follows from Lemma~\ref{lem:algorithm2-well-defined}. By steps~1 and~5 of Algorithm~\ref{alg:scsc}, one has $\bar u^1=u^1$ and
\begin{align}
&\lambda_i=4^{i-1}\lambda_1\qquad\forall1\le i\le k,\label{eq:algorithm2-regularization}\\
&\lambda_i\bar u^i=4\lambda_{i-1}\frac{\bar u^{i-1}+3u^i}{4}=\lambda_{i-1}\bar u^{i-1}+(\lambda_i-\lambda_{i-1})u^i\qquad\forall2\le i\le k.\label{eq:algorithm2-center}
\end{align}
By the definition of $f_i$ in step~3 and the KKT condition at $u_*^i$, one has
\begin{equation}\label{eq:algorithm2-subproblem-kkt}
0\in(F+B)(u_*^i)-\lambda_i(\bar u^i-u_*^i)\quad\Rightarrow\quad\lambda_i(\bar u^i-u_*^i)\in(F+B)(u_*^i)\qquad\forall1\le i\le k.
\end{equation}
Using $\bar u^1=u^1$, $0\in(F+B)(u^*)$, and~\eqref{eq:algorithm2-subproblem-kkt} at $i=1$, we obtain
\begin{equation}\label{eq:algorithm2-first-distance}
0\le2\ip{u^1-u_*^1}{u_*^1-u^*}=\norm{u^1-u^*}^2-\norm{u^1-u_*^1}^2-\norm{u_*^1-u^*}^2\quad\Rightarrow\quad\norm{u^1-u_*^1}\le\norm{u^1-u^*},
\end{equation}
where the first inequality follows from the monotonicity of $F+B$ and $\lambda_1>0$. For $1\le i\le k$, applying Theorem~\ref{thm:contract} to the call in step~3 with initial point $u^i$ and strong convexity parameter $\mu_x+\lambda_i$, we obtain
\begin{equation}\label{eq:algorithm2-call-bounds}
\norm{u^{i+1}-u_*^i}\le\frac18\norm{u^i-u_*^i},\qquad\norm{r^{i+1}}\le\frac{\mu_x+\lambda_i}{8}\norm{u^i-u_*^i}.
\end{equation}
For $2\le i\le k$, it follows from~\eqref{eq:algorithm2-center} and~\eqref{eq:algorithm2-subproblem-kkt} that
\[
(\lambda_i-\lambda_{i-1})(u^i-u_*^i)\in(F+B)(u_*^i)+\lambda_{i-1}(u_*^i-\bar u^{i-1}).
\]
In addition,~\eqref{eq:algorithm2-subproblem-kkt} at $i-1$ implies $0\in(F+B)(u_*^{i-1})+\lambda_{i-1}(u_*^{i-1}-\bar u^{i-1})$. By these inclusions, the monotonicity of $u\mapsto(F+B)(u)+\lambda_{i-1}(u-\bar u^{i-1})$, and $\lambda_i>\lambda_{i-1}$, one has
\[
0\le2\ip{u^i-u_*^i}{u_*^i-u_*^{i-1}}=\norm{u^i-u_*^{i-1}}^2-\norm{u^i-u_*^i}^2-\norm{u_*^i-u_*^{i-1}}^2.
\]
Using this and~\eqref{eq:algorithm2-call-bounds} at $i-1$, we obtain
\[
\norm{u^i-u_*^i}\le\norm{u^i-u_*^{i-1}}\le\frac18\norm{u^{i-1}-u_*^{i-1}}\qquad\forall2\le i\le k.
\]
By~\eqref{eq:algorithm2-first-distance}, the first bound in~\eqref{eq:algorithm2-stage-errors} below holds at $i=1$. Suppose that it holds at some $i<k$. Using the displayed inequality at $i+1$, we obtain $\norm{u^{i+1}-u_*^{i+1}}\le\norm{u^i-u_*^i}/8\le8^{-i}\norm{u^1-u^*}$. Hence the induction is completed. Using this and~\eqref{eq:algorithm2-call-bounds}, we obtain
\begin{equation}\label{eq:algorithm2-stage-errors}
\norm{u^i-u_*^i}\le8^{-(i-1)}\norm{u^1-u^*},\qquad\norm{u^{i+1}-u_*^i}\le8^{-i}\norm{u^1-u^*}\qquad\forall1\le i\le k.
\end{equation}
In addition, by~\eqref{eq:algorithm2-call-bounds} and the first bound in~\eqref{eq:algorithm2-stage-errors}, one has
\begin{equation}\label{eq:algorithm2-stage-residual}
\norm{r^{i+1}}\le(\mu_x+\lambda_i)8^{-i}\norm{u^1-u^*}\qquad\forall1\le i\le k.
\end{equation}

We next bound the accumulated regularization term. By~\eqref{eq:algorithm2-center}, one has, for $2\le i\le k$,
\[
\lambda_i(\bar u^i-u_*^i)-\lambda_{i-1}(\bar u^{i-1}-u_*^{i-1})=\lambda_i(u^i-u_*^i)-\lambda_{i-1}(u^i-u_*^{i-1}).
\]
Summing this equality over $i=2,\ldots,k$ and using $\bar u^1=u^1$, we obtain
\begin{equation}\label{eq:algorithm2-bias}
\begin{aligned}[b]
&\norm{\lambda_k(\bar u^k-u_*^k)}=\left\|\lambda_1(u^1-u_*^1)+\sum_{i=2}^k\left(\lambda_i(u^i-u_*^i)-\lambda_{i-1}(u^i-u_*^{i-1})\right)\right\|\\
&\le\lambda_1\norm{u^1-u_*^1}+\sum_{i=2}^k\left(\lambda_i\norm{u^i-u_*^i}+\lambda_{i-1}\norm{u^i-u_*^{i-1}}\right)\\
&\le\left(\lambda_1+\sum_{i=2}^k(\lambda_i+\lambda_{i-1})8^{-(i-1)}\right)\norm{u^1-u^*}=\left(1+\frac54\sum_{i=2}^k2^{-(i-1)}\right)\lambda_1\norm{u^1-u^*}\le\frac94\lambda_1\norm{u^1-u^*},
\end{aligned}
\end{equation}
where the first inequality follows from the triangle inequality, the second inequality follows from~\eqref{eq:algorithm2-first-distance} and the two bounds in~\eqref{eq:algorithm2-stage-errors} at $i$ and $i-1$, respectively, the second equality follows from $\lambda_i=4^{i-1}\lambda_1$, and the last inequality follows from $\sum_{i=2}^k2^{-(i-1)}\le1$. The sums are empty when $k=1$. Finally, by step~4 of Algorithm~\ref{alg:scsc}, one has
\begin{align*}
&\norm{e^{k+1}}=\norm{r^{k+1}-\lambda_k(u^{k+1}-u_*^k)+\lambda_k(\bar u^k-u_*^k)}\\
&\le\norm{r^{k+1}}+\lambda_k\norm{u^{k+1}-u_*^k}+\norm{\lambda_k(\bar u^k-u_*^k)}\le\left(\frac94\lambda_1+(\mu_x+2\lambda_k)8^{-k}\right)\norm{u^1-u^*},
\end{align*}
where the first inequality follows from the triangle inequality and the second follows from~\eqref{eq:algorithm2-stage-errors},~\eqref{eq:algorithm2-stage-residual}, and~\eqref{eq:algorithm2-bias}. Hence the conclusion of this lemma holds.
\end{proof}

\begin{proof}[Proof of Theorem~\ref{thm:scsc}]
Let $u^*$ be the unique saddle point of~\eqref{eq:scsc}, and let $F$ and $B$ be defined in~\eqref{eq:scsc-operators}. By Lemma~\ref{lem:algorithm2-well-defined}, Algorithm~\ref{alg:scsc} is well-defined and $e^{k+1}\in(F+B)(u^{k+1})$ at every reached iteration, including the preliminary output at $k=0$. By the stopping conditions in steps~1 and~5, its output also satisfies $\norm{e^{k+1}}\le\epsilon$ whenever the algorithm terminates.

We next show that the algorithm terminates. Suppose that it does not terminate in step~1. For any reached regularized iteration $k$ with $\lambda_k\ge\mu_x$, using~\eqref{eq:algorithm2-initial-distance},~\eqref{eq:stage-residual}, and~\eqref{eq:algorithm2-regularization}, we obtain
\begin{align*}
&\norm{e^{k+1}}\le\left(\frac94\lambda_1+3\lambda_k8^{-k}\right)\norm{u^1-u^*}=\left(\frac94+\frac38\,2^{-(k-1)}\right)\lambda_1\norm{u^1-u^*}\\
&\le\frac{21}{8}\lambda_1\norm{u^1-u^*}\le\frac{3}{8}\lambda_1\norm{u^1-u^0}\le\frac38\epsilon,
\end{align*}
where the first inequality also uses $\lambda_k\ge\mu_x$, the equality follows from $\lambda_k=4^{k-1}\lambda_1$, the second inequality follows from $k\ge1$, the third follows from~\eqref{eq:algorithm2-initial-distance}, and the last follows from the definition of $\lambda_1$ in step~1. Thus the stopping condition in step~5 holds whenever $\lambda_k\ge\mu_x$. In addition, every call terminates by Lemma~\ref{lem:algorithm2-well-defined}, and~\eqref{eq:algorithm2-regularization} implies $\lambda_k\ge\mu_x$ for all sufficiently large $k$ if the algorithm does not terminate. Hence Algorithm~\ref{alg:scsc} terminates no later than the first iteration with $\lambda_k\ge\mu_x$.

We next prove the evaluation bound. For a call of Algorithm~\ref{alg:contract} with regularization parameter $\lambda\ge0$, the supplied smoothness and curvature parameters are $(L_{\nabla f}+\lambda,\mu_x+\lambda,\mu_y+\lambda)$. Each evaluation of the gradient of the regularized smooth part requires one evaluation of $\nabla f$. By this and~\eqref{eq:algorithm1-cost}, the numbers of evaluations of $\nabla f$, $\prox_p$, and $\prox_q$ for this call are bounded, respectively, by
\begin{equation}\label{eq:algorithm2-call-cost}
129(N_{\mcA}+256\mathcal C_{\mcA}+3)\frac{L_{\nabla f}+\lambda}{\sqrt{(\mu_x+\lambda)(\mu_y+\lambda)}},
\end{equation}
where we used $\lceil a\rceil\le a+1$ and $L_{\nabla f}\ge\mu_x\ge\mu_y>0$. The preliminary call corresponds to $\lambda=0$, so~\eqref{eq:scsc-cost} holds if the algorithm terminates in step~1. Otherwise, let $k\ge1$ be its terminating regularized iteration.

By~\eqref{eq:algorithm2-regularization}, there are at most $\lceil(\log_4(\mu_y/\lambda_1))_+\rceil$ iterations with $\lambda_i<\mu_y$. Since $\mu_y\le L_{\nabla f}$, one has, at each such iteration,
\[
\frac{L_{\nabla f}+\lambda_i}{\sqrt{(\mu_x+\lambda_i)(\mu_y+\lambda_i)}}\le\frac{2L_{\nabla f}}{\sqrt{\mu_x\mu_y}}.
\]
For the iterations with $\mu_y\le\lambda_i<\mu_x$, we have
\begin{equation}\label{eq:algorithm2-small-lambda-cost}
\frac{L_{\nabla f}+\lambda_i}{\sqrt{(\mu_x+\lambda_i)(\mu_y+\lambda_i)}}\le\frac{2L_{\nabla f}}{\sqrt{\mu_x\lambda_i}},
\end{equation}
where the inequality follows from $\lambda_i<\mu_x\le L_{\nabla f}$, $\mu_x+\lambda_i\ge\mu_x$, and $\mu_y+\lambda_i\ge\lambda_i$. By~\eqref{eq:algorithm2-regularization}, the right-hand side of~\eqref{eq:algorithm2-small-lambda-cost} decreases by a factor of two at each successive iteration. Its sum over these iterations is at most $4L_{\nabla f}/\sqrt{\mu_x\mu_y}$, where we used $\lambda_i\ge\mu_y$ at the first such iteration. In addition, the termination argument implies that at most one reached iteration satisfies $\lambda_i\ge\mu_x$. At this iteration, one has
\[
\frac{L_{\nabla f}+\lambda_i}{\sqrt{(\mu_x+\lambda_i)(\mu_y+\lambda_i)}}\le\frac{L_{\nabla f}}{\lambda_i}+1\le\frac{L_{\nabla f}}{\mu_x}+1\le\frac{2L_{\nabla f}}{\sqrt{\mu_x\mu_y}},
\]
where the first inequality follows from $\sqrt{(\mu_x+\lambda_i)(\mu_y+\lambda_i)}\ge\lambda_i$, the second follows from $\lambda_i\ge\mu_x$, and the last follows from $L_{\nabla f}\ge\mu_x\ge\mu_y$. Adding the bounds over the three ranges, we obtain
\begin{equation}\label{eq:geometric-cost}
\sum_{i=1}^k\frac{L_{\nabla f}+\lambda_i}{\sqrt{(\mu_x+\lambda_i)(\mu_y+\lambda_i)}}\le\left(2\left\lceil\left(\log_4\frac{\mu_y}{\lambda_1}\right)_+\right\rceil+6\right)\frac{L_{\nabla f}}{\sqrt{\mu_x\mu_y}}\le\left(8+2\left(\log_4\frac{\mu_y}{\lambda_1}\right)_+\right)\frac{L_{\nabla f}}{\sqrt{\mu_x\mu_y}},
\end{equation}
where the second inequality follows from $\lceil t\rceil\le t+1$. In addition, by step~1 of Algorithm~\ref{alg:scsc} and~\eqref{eq:algorithm2-initial-distance}, one has
\[
\frac{\mu_y}{\lambda_1}=\frac{\mu_y\norm{u^1-u^0}}\epsilon\le\frac{9\mu_y\norm{u^0-u^*}}{8\epsilon}.
\]
Using this,~\eqref{eq:algorithm2-call-cost}, and~\eqref{eq:geometric-cost}, and including the preliminary call, we obtain that the numbers of evaluations of $\nabla f$, $\prox_p$, and $\prox_q$ are bounded, respectively, by
\begin{align*}
&129(N_{\mcA}+256\mathcal C_{\mcA}+3)\left(9+2\log_4\!\left(1+\frac{9\mu_y\norm{u^0-u^*}}{8\epsilon}\right)\right)\frac{L_{\nabla f}}{\sqrt{\mu_x\mu_y}}\\
&\le2^{11}(N_{\mcA}+256\mathcal C_{\mcA}+3)\frac{L_{\nabla f}}{\sqrt{\mu_x\mu_y}}\left(1+\log_4\!\left(1+\frac{\mu_y\norm{u^0-u^*}}{\epsilon}\right)\right),
\end{align*}
where the first bound follows from $(\log_4 a)_+\le\log_4(1+a)$ for $a>0$, and the inequality follows from $1+9a/8\le9(1+a)/8$ for $a\ge0$ and $129(9+2\log_4(9/8))<2^{11}$. Hence~\eqref{eq:scsc-cost} holds and the proof is completed.
\end{proof}

\subsection{Proof of the main result in Subsection~\ref{sec:pf-scsc}}\label{subsec:proof-algorithm3}

In this subsection we provide a proof of Theorem~\ref{thm:pf-scsc}. We first establish the residual inclusion and evaluation bound for each trial of Algorithm~\ref{alg:pf-scsc}.

\begin{lemma}\label{lem:pf-trial}
Suppose Assumptions~\ref{ass:scsc} and~\ref{ass:subsolver} hold. Let $F$ and $B$ be defined in~\eqref{eq:scsc-operators}. For any reached trial $t$ of Algorithm~\ref{alg:pf-scsc}, let $(u^t,r^t)$ be the pair generated by the algorithm and let $U_t$, $\alpha_t$, and $\nu_t$ be defined in step~5. Then any output $(\tilde u^{t+1},e^{t+1})$ of the call in step~6 satisfies $e^{t+1}\in(F+B)(\tilde u^{t+1})$. Moreover, if
\begin{equation}\label{eq:pf-valid}
U_t\ge L_{\nabla f},\qquad 0<\nu_t\le\mu_y,\qquad \nu_t\le\alpha_t\le\mu_x,
\end{equation}
then the trial completes within its evaluation limits and returns a pair satisfying
\begin{equation}\label{eq:pf-valid-reduction}
\norm{e^{t+1}}\le\max\{\epsilon,\norm{r^t}/8\}.
\end{equation}
\end{lemma}
\begin{proof}
Fix any reached trial $t$. By the stopping tests and updates in Algorithm~\ref{alg:pf-scsc}, one has $\norm{r^t}>\epsilon>0$. The execution requirements on $\mcA$ and the residual corrections in~\eqref{eq:algorithm1-updates} and step~4 of Algorithm~\ref{alg:scsc} imply $e^{t+1}\in(F+B)(\tilde u^{t+1})$ whenever the call returns.

Suppose that~\eqref{eq:pf-valid} holds, and let $u^*$ be the saddle point of~\eqref{eq:scsc}. Since $r^t\in(F+B)(u^t)$ and $0\in(F+B)(u^*)$, the $\nu_t$-strong monotonicity of $F+B$ implies $\nu_t\norm{u^t-u^*}\le\norm{r^t}$. By Theorem~\ref{thm:scsc} with parameters $(U_t,\alpha_t,\nu_t)$, initial point $u^t$, and tolerance $\max\{\epsilon,\norm{r^t}/8\}$, the call in step~6 without evaluation limits terminates with a pair satisfying~\eqref{eq:pf-valid-reduction}. In addition, by~\eqref{eq:scsc-cost}, its numbers of evaluations of $\nabla f$, $\prox_p$, and $\prox_q$ are bounded, respectively, by
\begin{align*}
&2^{11}(N_{\mcA}+256\mathcal C_{\mcA}+3)\frac{U_t}{\sqrt{\alpha_t\nu_t}}\left(1+\log_4\!\left(1+\frac{\nu_t\norm{u^t-u^*}}{\max\{\epsilon,\norm{r^t}/8\}}\right)\right)\\
&\le2^{11}(N_{\mcA}+256\mathcal C_{\mcA}+3)\frac{U_t}{\sqrt{\alpha_t\nu_t}}(1+\log_4 9)
<\widetilde{\mcC}_{\mcA}\frac{U_t}{\sqrt{\alpha_t\nu_t}},
\end{align*}
where the first inequality follows from $\nu_t\norm{u^t-u^*}\le\norm{r^t}$ and $\max\{\epsilon,\norm{r^t}/8\}\ge\norm{r^t}/8$, and the last inequality follows from $\log_4 9<2$ and the definition of $\widetilde{\mcC}_{\mcA}$ in step~6 of Algorithm~\ref{alg:pf-scsc}. Thus the call completes within its evaluation limits. Hence the conclusion follows.
\end{proof}

We now prove Theorem~\ref{thm:pf-scsc}.

\begin{proof}[Proof of Theorem~\ref{thm:pf-scsc}]
Let $u^*$ be the saddle point of~\eqref{eq:scsc}. By step~1 of Algorithm~\ref{alg:pf-scsc}, one has
\begin{align*}
&\mu_x\norm d^2
\le\ip{\nabla_x f(u_x^0+d,u_y^0)-\nabla_x f(u_x^0,u_y^0)}d\le\norm{\nabla_x f(u_x^0+d,u_y^0)-\nabla_x f(u_x^0,u_y^0)}\norm d
=\rho\norm d^2,
\end{align*}
where the first inequality follows from the $\mu_x$-strong convexity of $f(\cdot,u_y^0)$, the second follows from the Cauchy--Schwarz inequality, and the equality follows from the definition of $\rho$. Since $d\ne0$ and both points belong to $\mathcal X\times\mathcal Y$, we obtain $\mu_x\le\rho$. Using this, Assumption~\ref{ass:scsc}, and the Lipschitz continuity of $\nabla f$, we have $0<\mu_y\le\mu_x\le\rho\le L_{\nabla f}$.

By step~2 and the resolvent identity, one has
\[
\rho(u^0-u^1)-F(u^0)\in B(u^1),\qquad
r^1=\rho(u^0-u^1)+F(u^1)-F(u^0)\in(F+B)(u^1).
\]
Thus $u^1\in\mathcal X\times\mathcal Y$, and the initialization is well-defined. In addition, one has
\begin{align*}
\rho\norm{u^0-u^1}^2
&\le\ip{\rho(u^0-u^*)+F(u^0)-F(u^*)}{u^0-u^1}\\
&\le(\rho+L_{\nabla f})\norm{u^0-u^*}\norm{u^0-u^1},
\end{align*}
where the first inequality follows from $\rho(u^0-u^1)-F(u^0)\in B(u^1)$, $-F(u^*)\in B(u^*)$, and the monotonicity of $F$ and $B$, and the second follows from the Cauchy--Schwarz inequality and the Lipschitz continuity of $F$. Dividing by $\rho\norm{u^0-u^1}$ when it is positive, with the zero case immediate, we obtain $\norm{u^0-u^1}\le(1+L_{\nabla f}/\rho)\norm{u^0-u^*}$. Using this and the definition of $r^1$ in step~2, we obtain
\begin{equation}\label{eq:pf-initial-residual}
\norm{r^1}\le(\rho+L_{\nabla f})\norm{u^0-u^1}\le\frac{(\rho+L_{\nabla f})^2}{\rho}\norm{u^0-u^*}\le\frac{4L_{\nabla f}^2}{\mu_x}\norm{u^0-u^*},
\end{equation}
where the first inequality follows from the triangle inequality and the Lipschitz continuity of $F$, the second follows from $\norm{u^0-u^1}\le(1+L_{\nabla f}/\rho)\norm{u^0-u^*}$, and the last follows from $\mu_x\le\rho\le L_{\nabla f}$.

The computations in steps~1 and~2 use at most three evaluations of $\nabla f$ and one of each proximal mapping. If $\norm{r^1}\le\epsilon$, step~3 returns $u^1$, which is an $\epsilon$-stationary point by~Definition~\ref{def:epsilon-stationary}, and the conclusion follows. Suppose now that $\norm{r^1}>\epsilon$. The parameters and evaluation limits in every reached trial are positive and finite. By the execution requirements on $\mcA$, each call can therefore be executed up to its limit. Since $r^1\in(F+B)(u^1)$, Lemma~\ref{lem:pf-trial} and the updates in steps~9 and~11 imply $r^t\in(F+B)(u^t)$ at every reached trial $t$. Thus Algorithm~\ref{alg:pf-scsc} is well-defined.

We next prove finite termination. By steps~9 and~11, an unsuccessful trial advances to the next triple in $\mathcal I$, and a successful trial keeps the triple unchanged. Let
\[
i_*=\left\lceil\log_2\frac{L_{\nabla f}}{\rho}\right\rceil,\qquad
j_*=\left\lceil\log_2\frac{\rho}{\mu_x}\right\rceil,\qquad
k_*=\left\lceil\log_2\frac{\rho}{\mu_y}\right\rceil.
\]
By $\mu_y\le\mu_x\le\rho\le L_{\nabla f}$ and the definition of the ceiling function, one has $i_*\ge0$, $0\le j_*\le k_*$, and
\[
\rho 2^{i_*}\ge L_{\nabla f},\qquad
\rho 2^{-j_*}\le\mu_x,\qquad
\rho 2^{-k_*}\le\mu_y.
\]
It follows that $(i_*,j_*,k_*)\in\mathcal I$ and the corresponding parameters satisfy~\eqref{eq:pf-valid}. By Lemma~\ref{lem:pf-trial}, every trial with this triple either terminates in step~7 or passes the acceptance test in step~8. Hence no later triple is reached. Let $q_*=2i_*+j_*+k_*$. Using the definitions of $i_*$, $j_*$, and $k_*$ and $\lceil a\rceil\le a+1$, we obtain
\begin{align}
&q_*\le2\log_2\frac{L_{\nabla f}}{\rho}
+\log_2\frac{\rho}{\mu_x}+\log_2\frac{\rho}{\mu_y}+4=2\log_2\!\left(\frac{L_{\nabla f}}{\sqrt{\mu_x\mu_y}}\right)+4\quad\Rightarrow\quad 2^{q_*/2}\le\frac{4L_{\nabla f}}{\sqrt{\mu_x\mu_y}}.\label{eq:pf-level}
\end{align}
Suppose that $s$ successful trials precede a reached trial $t$. By the stopping and acceptance tests in steps~7 and~8 and the updates in steps~9 and~11, one has $\epsilon<\norm{r^t}\le2^{-s}\norm{r^1}$ and hence $s<\log_2(\norm{r^1}/\epsilon)$. Thus there are at most $\lceil\log_2(\norm{r^1}/\epsilon)\rceil$ successful trials, including the terminating trial. In addition, each unsuccessful trial advances to a new triple, and there are finitely many triples up to $(i_*,j_*,k_*)$ in $\mathcal I$. By these and the fact that no later triple is reached, Algorithm~\ref{alg:pf-scsc} terminates in finitely many trials. By Lemma~\ref{lem:pf-trial}, the stopping test in step~7, and~Definition~\ref{def:epsilon-stationary}, the returned point is an $\epsilon$-stationary point of~\eqref{eq:scsc}.

It remains to establish the evaluation bound. For a fixed level $q$, the indices satisfy $2i+j+k=q$ and $0\le j\le k$. For every $0\le i\le\lfloor q/2\rfloor$, there are $\lfloor q/2\rfloor-i+1$ choices of $j$, and each determines $k=q-2i-j$. Thus the number of triples at level $q$ satisfies
\begin{equation}\label{eq:pf-level-triples}
\sum_{i=0}^{\lfloor q/2\rfloor}(\lfloor q/2\rfloor-i+1)
=\frac{(\lfloor q/2\rfloor+1)(\lfloor q/2\rfloor+2)}2\le(q+2)^2.
\end{equation}
By step~5, every trial $t$ at level $q$ has $2i_t+j_t+k_t=q$ and $U_t/\sqrt{\alpha_t\nu_t}=2^{q/2}$. Using this, $q\ge0$, and the evaluation limit in step~6, we obtain that the numbers of evaluations of $\nabla f$, $\prox_p$, and $\prox_q$ in that trial are bounded, respectively, by $(\widetilde{\mcC}_{\mcA}+1)2^{q/2}$. Since there is at most one unsuccessful trial at each triple,~\eqref{eq:pf-level-triples} bounds the number of unsuccessful trials at level $q$. In addition, each successful trial has level at most $q_*$, and there are at most $\lceil\log_2(\norm{r^1}/\epsilon)\rceil$ such trials. Combining these bounds and the initialization cost, we obtain that the total numbers of evaluations of $\nabla f$, $\prox_p$, and $\prox_q$ are bounded, respectively, by
\begin{equation*}
\begin{aligned}[b]
&3+(\widetilde{\mcC}_{\mcA}+1)\left(\sum_{q=0}^{q_*}(q+2)^2 2^{q/2}+\left\lceil\log_2\frac{\norm{r^1}}\epsilon\right\rceil2^{q_*/2}\right)\\
&\le3+(\widetilde{\mcC}_{\mcA}+1)2^{q_*/2}\left((2+\sqrt2)(q_*+2)^2+\left\lceil\log_2\frac{\norm{r^1}}\epsilon\right\rceil\right)\\
&\le3+4(\widetilde{\mcC}_{\mcA}+1)\frac{L_{\nabla f}}{\sqrt{\mu_x\mu_y}}\left((2+\sqrt2)\left(2\log_2\frac{L_{\nabla f}}{\sqrt{\mu_x\mu_y}}+6\right)^2+\left\lceil\log_2\!\left(1+\frac{4L_{\nabla f}^2\norm{u^0-u^*}}{\mu_x\epsilon}\right)\right\rceil\right),
\end{aligned}
\end{equation*}
where the term $3$ accounts for the evaluations in steps~1 and~2, the first inequality follows from $(q+2)^2\le(q_*+2)^2$ for $q\le q_*$ and $\sum_{q=0}^{q_*}2^{-(q_*-q)/2}\le2+\sqrt2$, and the last inequality follows from~\eqref{eq:pf-initial-residual},~\eqref{eq:pf-level}, and the monotonicity of the logarithm and ceiling functions. Hence~\eqref{eq:pf-complexity} holds, and the conclusion follows.

\end{proof}
\subsection{Proof of the main result in Subsections~\ref{sec:ncc} and~\ref{sec:pf-ncc}}\label{subsec:proof-algorithm4}.

In this subsection we first establish four lemmas and then prove Theorems~\ref{thm:ncc} and \ref{thm:pf-ncc}. Accepted trials are defined before Algorithm~\ref{alg:ncc}. In particular, an accepted trial updates the primal center, inner initial point, and stored residual in step~7 without terminating or changing the smoothness estimate. For $\mu>0$, define
\begin{equation}\label{eq:ncc-regularized-objective}
H_\mu(x,y)=H(x,y)-\frac\mu2\norm{y-y^0}^2.
\end{equation}
By Assumption~\ref{ass:ncc}(i)--(iii), the maxima of $H(x,\cdot)$ and $H_\mu(x,\cdot)$ are finite and attained for every $x\in\mathcal X$.

\begin{lemma}\label{lem:ncc-trial}
Suppose that Assumption~\ref{ass:ncc} holds. For the sequence generated by Algorithm~\ref{alg:ncc}, let $\tau_t$ be defined in step~4. Then the algorithm is well-defined at every reached trial $t$, and its call in step~4 returns or is interrupted after finitely many evaluations. If the call returns, the returned point is a $\tau_t$-stationary point of the subproblem in step~4.
\end{lemma}
\begin{proof}
We prove by induction that, at the beginning of every reached trial $t$, one has $x^t\in\mathcal X$, $(\hat x,\hat y)\in\mathcal X\times\mathcal Y$, and
\begin{equation}\label{eq:ncc-stored-residual}
r^t\in\begin{pmatrix}\nabla_x h_t(\hat x,\hat y)\\-\nabla_y h_t(\hat x,\hat y)\end{pmatrix}+B(\hat x,\hat y),
\end{equation}
where $(\hat x,\hat y)$ denotes the inner initial point at that trial. By step~2 and the resolvent identity,
\[
\beta_1\bigl((x^0,y^0)-(\hat x,\hat y)\bigr)-A(x^0,y^0)\in B(\hat x,\hat y).
\]
Together with $x^1=x^0$,~\eqref{eq:h-k}, and the definition of $r^1$ in step~2, this proves the induction hypothesis at $t=1$.

Suppose that the hypothesis holds at a reached trial $t$. By the parameter updates, $0<\mu\le\epsilon/32\le\beta_t$ and $\tau_t>0$. Thus the bound in~\eqref{eq:ncc-trial-budget} is finite. The execution requirements on $\mcA$ ensure that the call in step~4 is well-defined and returns or is interrupted within this bound. If it returns, its residual corrections and stopping tests imply that $\norm{e^{t+1}}\le\tau_t$ and, before any reset in step~7,
\begin{equation}\label{eq:ncc-subproblem-kkt}
e^{t+1}\in\begin{pmatrix}\nabla_x h_t(x^{t+1},y^{t+1})\\-\nabla_y h_t(x^{t+1},y^{t+1})\end{pmatrix}+B(x^{t+1},y^{t+1}).
\end{equation}
Hence $(x^{t+1},y^{t+1})\in\mathcal X\times\mathcal Y$ is a $\tau_t$-stationary point of the subproblem in step~4.

If the trial is accepted, the center changes to $x^{t+1}$ and the inner initial point changes to $(x^{t+1},y^{t+1})$. Subtracting $\bigl(2\beta_t(x^{t+1}-x^t),0\bigr)$ from~\eqref{eq:ncc-subproblem-kkt} proves~\eqref{eq:ncc-stored-residual} at trial $t+1$. If step~7 rejects the trial, the center and inner initial point are retained, while $\beta_{t+1}=2\beta_t$. Adding $\bigl(2\beta_t(\hat x-x^t),0\bigr)$ to~\eqref{eq:ncc-stored-residual} proves the same inclusion at trial $t+1$. The feasibility statements are also preserved. This completes the induction and proves the conclusion.
\end{proof}

The next two lemmas bound the decrease at an accepted trial and its sum within a run, respectively.

\begin{lemma}\label{lem:ncc-descent}
Suppose Assumptions~\ref{ass:ncc} and~\ref{ass:subsolver} hold. For the sequence generated by Algorithm~\ref{alg:ncc}, let $\mu$ and $\tau_t$ be defined in steps~1 and~4, respectively, and let $H_\mu$ be defined in~\eqref{eq:ncc-regularized-objective}. Then, at every accepted trial $t$, one has
\begin{equation}\label{eq:nonterminal-descent}
\max_y H_\mu(x^t,y)-\max_y H_\mu(x^{t+1},y)\ge\beta_t\norm{x^{t+1}-x^t}^2>\frac{\epsilon^2}{16\beta_t}.
\end{equation}
\end{lemma}
\begin{proof}
Fix any accepted trial $t$. By Lemma~\ref{lem:ncc-trial} and step~6 of Algorithm~\ref{alg:ncc}, one has $\norm{e^{t+1}}\le\tau_t$ and $\norm{x^{t+1}-x^t}>\epsilon/(4\beta_t)$. By~\eqref{eq:ncc-subproblem-kkt}, $e_y^{t+1}\in\partial(-H_\mu(x^{t+1},\cdot))(y^{t+1})$. Using this and the $\mu$-strong convexity of $-H_\mu(x^{t+1},\cdot)$, we obtain that for every $y\in\mathcal Y$,
\[
\begin{aligned}
H_\mu(x^{t+1},y)
&\le H_\mu(x^{t+1},y^{t+1})-\ip{e_y^{t+1}}{y-y^{t+1}}-\frac\mu2\norm{y-y^{t+1}}^2\\
&\le H_\mu(x^{t+1},y^{t+1})+\frac{\norm{e_y^{t+1}}^2}{2\mu}.
\end{aligned}
\]
Taking the maximum over $y$ and using the convexity of $p$ and~\eqref{eq:ncc-subproblem-kkt}, we obtain
\[
\begin{aligned}
&\max_y H_\mu(x^t,y)-\max_y H_\mu(x^{t+1},y)\\
&\ge h(x^t,y^{t+1})-h(x^{t+1},y^{t+1})+p(x^t)-p(x^{t+1})-\frac{\norm{e_y^{t+1}}^2}{2\mu}\\
&\ge h(x^t,y^{t+1})-h(x^{t+1},y^{t+1})+\ip{\nabla_x h(x^{t+1},y^{t+1})-e_x^{t+1}}{x^{t+1}-x^t}
+2\beta_t\norm{x^{t+1}-x^t}^2-\frac{\norm{e_y^{t+1}}^2}{2\mu}\\
&\ge\frac{3\beta_t}{2}\norm{x^{t+1}-x^t}^2-\tau_t\norm{x^{t+1}-x^t}-\frac{\tau_t^2}{2\mu}\\
&>\left(\frac32-\frac1{16}-\frac1{512}\right)\beta_t\norm{x^{t+1}-x^t}^2
\ge\beta_t\norm{x^{t+1}-x^t}^2>\frac{\epsilon^2}{16\beta_t},
\end{aligned}
\]
where the third inequality follows from~\eqref{eq:ncc-model-test} and $\norm{e^{t+1}}\le\tau_t$, and the fourth follows from $\tau_t\le\epsilon/64$, $\tau_t^2/(2\mu)=\epsilon^2/(8192\beta_t)$, and $\norm{x^{t+1}-x^t}>\epsilon/(4\beta_t)$. Hence,~\eqref{eq:nonterminal-descent} holds.
\end{proof}

\begin{lemma}\label{lem:ncc-phase}
Suppose Assumptions~\ref{ass:ncc} and~\ref{ass:subsolver} hold. For the sequence generated by Algorithm~\ref{alg:ncc}, let $\mathcal T$ be the set of accepted trials up to any reached trial. Then
\begin{equation}\label{eq:ncc-accepted-energy}
\sum_{t\in\mathcal T}\beta_t\norm{x^{t+1}-x^t}^2\le\Delta_0+\frac{\epsilon D_{\mathbf y}}8,
\end{equation}
where $\Delta_0$ is defined in Theorem~\ref{thm:ncc} and $D_{\mathbf y}$ is given in Assumption~\ref{ass:ncc}(ii). Moreover, any consecutive trials with fixed $\beta_t$ contain at most $1+16\beta_t(\Delta_0+\epsilon D_{\mathbf y}/8)/\epsilon^2$ calls to Algorithm~\ref{alg:scsc}.
\end{lemma}
\begin{proof}
Fix any accepted trial $t$. By~\eqref{eq:ncc-subproblem-kkt}, one has
\[
e_y^{t+1}-\mu(y^{t+1}-y^0)\in\partial q(y^{t+1})-\nabla_y h(x^{t+1},y^{t+1}).
\]
By Lemma~\ref{lem:ncc-trial} and steps~1,~4, and~5 of Algorithm~\ref{alg:ncc}, $\norm{e^{t+1}}\le\tau_t$, $\mu\le\beta_t$, and $\norm{y^{t+1}-y^0}\le D$. Thus
\[
\norm{e_y^{t+1}-\mu(y^{t+1}-y^0)}\le\tau_t+\mu D
=\frac\epsilon{64}\sqrt{\frac\mu{\beta_t}}+\frac\epsilon{32}\le\frac{3\epsilon}{64}.
\]
By the convexity of $q-h(x^{t+1},\cdot)$, one has
\[
H(x^{t+1},y)\le H(x^{t+1},y^{t+1})+\ip{\mu(y^{t+1}-y^0)-e_y^{t+1}}{y-y^{t+1}},\qquad\forall y\in\mathcal Y.
\]
Using these inequalities and $\norm{y-y^{t+1}}\le2D_{\mathbf y}$ from Assumption~\ref{ass:ncc}(ii), we obtain
\begin{equation}\label{eq:ncc-dual-comparison}
\begin{aligned}[b]
&\max_y H(x^{t+1},y)\le H(x^{t+1},y^{t+1})+\frac{3\epsilon D_{\mathbf y}}{32}\\
&\le\max_y H_\mu(x^{t+1},y)+\frac\mu2\norm{y^{t+1}-y^0}^2+\frac{3\epsilon D_{\mathbf y}}{32}\le\max_y H_\mu(x^{t+1},y)+\frac{\epsilon D_{\mathbf y}}8,
\end{aligned}
\end{equation}
where the second inequality follows from~\eqref{eq:ncc-regularized-objective}, and the last follows from $\mu\norm{y^{t+1}-y^0}^2/2\le\mu D D_{\mathbf y}=\epsilon D_{\mathbf y}/32$, by the radius test and Assumption~\ref{ass:ncc}(ii).

If $\mathcal T=\varnothing$,~\eqref{eq:ncc-accepted-energy} follows from $\Delta_0\ge0$. Suppose now that $\mathcal T\ne\varnothing$, and let $a=\max\mathcal T$. Since $\mu$ is fixed, $x^1=x^0$, and step~7 implies $x^{k+1}=x^k$ at every rejected trial, summing~\eqref{eq:nonterminal-descent} over $t\in\mathcal T$ yields
\[
\begin{aligned}
&\sum_{t\in\mathcal T}\beta_t\norm{x^{t+1}-x^t}^2
\le\max_y H_\mu(x^0,y)-\max_y H_\mu(x^{a+1},y)\\
&\le\max_y H(x^0,y)-\max_y H(x^{a+1},y)+\frac{\epsilon D_{\mathbf y}}8\le\max_y H(x^0,y)-H^*+\frac{\epsilon D_{\mathbf y}}8
=\Delta_0+\frac{\epsilon D_{\mathbf y}}8,
\end{aligned}
\]
where the second inequality follows from $H_\mu\le H$ and~\eqref{eq:ncc-dual-comparison}, and the last inequality and equality follow from the definitions of $H^*$ and $\Delta_0$, respectively. Hence,~\eqref{eq:ncc-accepted-energy} holds.

By Lemma~\ref{lem:ncc-descent} and~\eqref{eq:ncc-accepted-energy}, the number of accepted trials with fixed $\beta_t$ is at most $16\beta_t(\Delta_0+\epsilon D_{\mathbf y}/8)/\epsilon^2$. Any other trial either terminates the run or doubles $\beta_t$. Thus consecutive trials with fixed $\beta_t$ contain at most one additional trial, which proves the remaining conclusion.
\end{proof}

\begin{lemma}\label{lem:ncc-run-cost}
Suppose Assumptions~\ref{ass:ncc} and~\ref{ass:subsolver} hold, and $0<\epsilon<32L_{\nabla h}$. Let Algorithm~\ref{alg:ncc} be run with $1\le D\le4D_{\mathbf y}$. Then the algorithm terminates with an $\epsilon$-stationary point of~\eqref{eq:ncc} or with failure, and failure occurs only if $D<2D_{\mathbf y}$.
\end{lemma}
\begin{proof}
Fix any reached trial $t$ with $\beta_t\ge L_{\nabla h}$. By Assumption~\ref{ass:ncc}(iii),~\eqref{eq:h-k}, and $\mu\le\beta_t$, the subproblem in step~4 is $\beta_t$-strongly convex in $x$ and $\mu$-strongly concave in $y$, and $\nabla h_t$ is $3\beta_t$-Lipschitz continuous. Let $(x_*^t,y_*^t)$ be its unique saddle point and $(\hat x,\hat y)$ its initial point. By~\eqref{eq:ncc-stored-residual}, strong monotonicity, one has
\[
\beta_t\norm{\hat x-x_*^t}^2+\mu\norm{\hat y-y_*^t}^2
\le\ip{r^t}{(\hat x,\hat y)-(x_*^t,y_*^t)}\le\left(\frac{\norm{r_x^t}^2}{\beta_t}+\frac{\norm{r_y^t}^2}\mu\right)^{1/2}
\left(\beta_t\norm{\hat x-x_*^t}^2+\mu\norm{\hat y-y_*^t}^2\right)^{1/2}.
\]
Using this, $\mu\le\beta_t$, and $\tau_t=(\epsilon/64)\sqrt{\mu/\beta_t}$, we obtain
\[
\frac\mu{\tau_t}\norm{(\hat x,\hat y)-(x_*^t,y_*^t)}
\le\frac{64}\epsilon\sqrt{\norm{r_x^t}^2+\frac{\beta_t}\mu\norm{r_y^t}^2}.
\]
By Theorem~\ref{thm:scsc} with $(L_{\nabla f},\mu_x,\mu_y)=(3\beta_t,\beta_t,\mu)$ and $3\cdot2^{11}(N_{\mcA}+256\mathcal C_{\mcA}+3)\le\widetilde{\mcC}_{\mcA}$, the call finishes within the limit~\eqref{eq:ncc-trial-budget}. In addition, the $L_{\nabla h}$-weak convexity of $h(\cdot,y^{t+1})$ implies
\[
\begin{aligned}
&h(x^t,y^{t+1})-h(x^{t+1},y^{t+1})+\ip{\nabla_x h(x^{t+1},y^{t+1})}{x^{t+1}-x^t}
+\frac{\beta_t}{2}\norm{x^{t+1}-x^t}^2\ge\frac{\beta_t-L_{\nabla h}}2\norm{x^{t+1}-x^t}^2\ge0.
\end{aligned}
\]
Thus~\eqref{eq:ncc-model-test} holds, so the trial either terminates the run or is accepted. Since $\beta_1=\epsilon/32<L_{\nabla h}$, step~7 can double $\beta_t$ only when $\beta_t<L_{\nabla h}$. Together with $\mu=\epsilon/(32D)$ and $1\le D\le4D_{\mathbf y}$, this implies
\begin{equation}\label{eq:ncc-search-bound}
\frac\epsilon{128D_{\mathbf y}}\le\mu\le\beta_1=\frac\epsilon{32}\le\beta_t<2L_{\nabla h}.
\end{equation}
By~\eqref{eq:ncc-search-bound}, only finitely many rejected trials occur. Lemmas~\ref{lem:ncc-trial} and~\ref{lem:ncc-phase} ensure that each call is finite and that only finitely many consecutive trials retain the same $\beta_t$. Hence the run terminates.

If the run returns a point in step~6, then~\eqref{eq:h-k} and~\eqref{eq:ncc-subproblem-kkt} imply
\[
e^{t+1}-\bigl(2\beta_t(x^{t+1}-x^t),\mu(y^{t+1}-y^0)\bigr)\in(A+B)(x^{t+1},y^{t+1}).
\]
By Lemma~\ref{lem:ncc-trial} and steps~5 and~6, the norm of this residual is at most
\[
\norm{e^{t+1}-\bigl(2\beta_t(x^{t+1}-x^t),\mu(y^{t+1}-y^0)\bigr)}\le\tau_t+2\beta_t\norm{x^{t+1}-x^t}+\mu\norm{y^{t+1}-y^0}
\le\frac\epsilon{64}+\frac\epsilon2+\frac\epsilon{32}=\frac{35\epsilon}{64}<\epsilon.
\]
Thus the returned point is an $\epsilon$-stationary point of~\eqref{eq:ncc} by Definition~\ref{def:epsilon-stationary}. If the run returns failure in step~5, Assumption~\ref{ass:ncc}(ii) and Lemma~\ref{lem:ncc-trial} imply $D<\norm{y^{t+1}-y^0}\le2D_{\mathbf y}$.

\end{proof}

\begin{proof}[Proof of Theorem~\ref{thm:ncc}]
We first bound the numbers of evaluations in a run of Algorithm~\ref{alg:ncc} with any $1\le D\le4D_{\mathbf y}$. By Lemma~\ref{lem:ncc-run-cost}, the run terminates.

The tests in steps~5--7 require at most two evaluations of $h$ and one additional evaluation of $\nabla h$. By~\eqref{eq:ncc-trial-budget}, $\beta_s\ge\mu$, and $\widetilde{\mcC}_{\mcA}\ge3$, the numbers of evaluations of $h$, $\nabla h$, $\prox_p$, and $\prox_q$ at every trial $s$, including its tests, are bounded, respectively, by
\begin{equation}\label{eq:ncc-trial-evaluations}
2\widetilde{\mcC}_{\mcA}\sqrt{\frac{\beta_s}\mu}
\left(1+\log_4\!\left(1+\frac{64}\epsilon\sqrt{\norm{r_x^s}^2+\frac{\beta_s}\mu\norm{r_y^s}^2}\right)\right).
\end{equation}
Fix a maximal group of consecutive trials with the same $\beta_t$, and let $t$ be its first index. We first bound the logarithm in~\eqref{eq:ncc-trial-evaluations} at trial $t$.

Suppose no trial before $t$ has been accepted. Then $x^t=x^0$ and $(\hat x,\hat y)$ is the point computed in step~2. Its optimality condition and the convexity of $p$ imply
\[
p(\hat x)-p(x^0)\le-\ip{\nabla_x h(x^0,y^0)}{\hat x-x^0}-\beta_1\norm{\hat x-x^0}^2.
\]
Using this, $0<\beta_1/L_{\nabla h}<1$, the convexity of $p$, and the Lipschitz continuity of $\nabla h$, we obtain that for every $y\in\mathcal Y$,
\[
\begin{aligned}
&H\!\left(x^0+\frac{\beta_1}{L_{\nabla h}}(\hat x-x^0),y\right)\le H(x^0,y)+\frac{\beta_1}{L_{\nabla h}}
\left(\ip{\nabla_x h(x^0,y)}{\hat x-x^0}+p(\hat x)-p(x^0)\right)
+\frac{\beta_1^2}{2L_{\nabla h}}\norm{\hat x-x^0}^2\\
&\le H(x^0,y)+\frac{\beta_1}{L_{\nabla h}}
\ip{\nabla_x h(x^0,y)-\nabla_x h(x^0,y^0)}{\hat x-x^0}
-\frac{\beta_1^2}{2L_{\nabla h}}\norm{\hat x-x^0}^2\\
&\le H(x^0,y)+2\beta_1D_{\mathbf y}\norm{\hat x-x^0}-\frac{\beta_1^2}{2L_{\nabla h}}\norm{\hat x-x^0}^2,
\end{aligned}
\]
where the last inequality follows from $\norm{y-y^0}\le2D_{\mathbf y}$ by Assumption~\ref{ass:ncc}(ii). Taking the maximum over $y$ and using the definitions of $H^*$ and $\Delta_0$, we obtain
\[
\frac{\beta_1^2}{2L_{\nabla h}}\norm{\hat x-x^0}^2-2\beta_1D_{\mathbf y}\norm{\hat x-x^0}\le\Delta_0.
\]
Solving this quadratic inequality and using $\norm{\hat y-y^0}\le2D_{\mathbf y}$ yields
\[
\begin{aligned}
\norm{(\hat x,\hat y)-(x^0,y^0)}
&\le2D_{\mathbf y}+\frac{2L_{\nabla h}D_{\mathbf y}+\sqrt{4L_{\nabla h}^2D_{\mathbf y}^2+2L_{\nabla h}\Delta_0}}{\beta_1}\le2D_{\mathbf y}+\frac{4L_{\nabla h}D_{\mathbf y}+\sqrt{2L_{\nabla h}\Delta_0}}{\beta_1}.
\end{aligned}
\]
Moreover, steps~2 and~7 imply
\[
r^t=A(\hat x,\hat y)-A(x^0,y^0)
+\bigl((2\beta_t-\beta_1)(\hat x-x^0),(\mu-\beta_1)(\hat y-y^0)\bigr).
\]
By~\eqref{eq:ncc-search-bound}, $\norm{r^t}\le(L_{\nabla h}+2\beta_t)\norm{(\hat x,\hat y)-(x^0,y^0)}\le5L_{\nabla h}\norm{(\hat x,\hat y)-(x^0,y^0)}$. Combining these inequalities, we obtain
\begin{equation}\label{eq:ncc-initial-residual-bound}
\begin{aligned}[b]
&\frac{64}\epsilon\sqrt{\norm{r_x^t}^2+\frac{\beta_t}\mu\norm{r_y^t}^2}\le\frac{320L_{\nabla h}}\epsilon\sqrt{\frac{\beta_t}\mu}
\left(2D_{\mathbf y}+\frac{4L_{\nabla h}D_{\mathbf y}+\sqrt{2L_{\nabla h}\Delta_0}}{\beta_1}\right).
\end{aligned}
\end{equation}

Suppose now that a trial before $t$ has been accepted, and let $i<t$ be the last such trial. Then $(\hat x,\hat y)=(x^{i+1},y^{i+1})$ and $x^t=x^{i+1}$. The residual correction in every intervening rejected trial is zero, so step~7 implies $r^t=e^{i+1}-\bigl(2\beta_i(x^{i+1}-x^i),0\bigr)$. Using this, $\norm{e^{i+1}}\le(\epsilon/64)\sqrt{\mu/\beta_i}$,~\eqref{eq:ncc-accepted-energy}, and~\eqref{eq:ncc-search-bound}, we obtain
\[
\begin{aligned}
\frac{64}\epsilon\sqrt{\norm{r_x^t}^2+\frac{\beta_t}\mu\norm{r_y^t}^2}
&\le\sqrt{\frac{\beta_t}{\beta_i}}+\frac{128\beta_i}\epsilon\norm{x^{i+1}-x^i}\le\sqrt{\frac{2L_{\nabla h}}{\beta_1}}
+\frac{128}\epsilon\sqrt{2L_{\nabla h}\left(\Delta_0+\frac{\epsilon D_{\mathbf y}}8\right)}.
\end{aligned}
\]
Using this,~\eqref{eq:ncc-search-bound},~\eqref{eq:ncc-initial-residual-bound}, and $D_{\mathbf y}\ge1$, we obtain in both cases
\begin{equation}\label{eq:ncc-first-log-bound}
\begin{aligned}[b]
&1+\log_4\!\left(1+\frac{64}\epsilon\sqrt{\norm{r_x^t}^2+\frac{\beta_t}\mu\norm{r_y^t}^2}\right)\\
&\le1+\log_4\!\left(2^{20}\left(1+\frac{L_{\nabla h}\Delta_0}{\epsilon^2}+\frac{L_{\nabla h}D_{\mathbf y}}\epsilon\right)^{5/2}\right)\\
&=11+\frac52\log_4\!\left(1+\frac{L_{\nabla h}\Delta_0}{\epsilon^2}+\frac{L_{\nabla h}D_{\mathbf y}}\epsilon\right)
\le11+\frac{2L_{\nabla h}\Delta_0}{\epsilon^2}+\frac{2L_{\nabla h}D_{\mathbf y}}\epsilon,
\end{aligned}
\end{equation}
where the last inequality follows from $\log(1+v)\le v$ for $v\ge0$. Combining~\eqref{eq:ncc-trial-evaluations} and~\eqref{eq:ncc-first-log-bound}, we obtain that the numbers of evaluations at trial $t$ are bounded, respectively, by
\begin{equation}\label{eq:ncc-first-trial-evaluations}
2\widetilde{\mcC}_{\mcA}\sqrt{\frac{\beta_t}\mu}
\left(11+\frac{2L_{\nabla h}\Delta_0}{\epsilon^2}+\frac{2L_{\nabla h}D_{\mathbf y}}\epsilon\right).
\end{equation}

Every later trial $s$ in this group follows an accepted trial $s-1$ with the same parameters. By step~7 and $\norm{e^s}\le\tau_t$, one has
\[
\frac{64}\epsilon\sqrt{\norm{r_x^s}^2+\frac{\beta_t}\mu\norm{r_y^s}^2}
\le1+\frac{128\beta_t}\epsilon\norm{x^s-x^{s-1}}.
\]
By~\eqref{eq:nonterminal-descent}, $\beta_t\norm{x^s-x^{s-1}}/\epsilon>1/4$. Using this,~\eqref{eq:ncc-trial-evaluations}, and $1+\log_4(2+128v)\le64v^2$ for $v\ge1/4$, we obtain that the numbers of evaluations at trial $s$ are bounded, respectively, by
\begin{equation}\label{eq:ncc-later-trial-evaluations}
128\widetilde{\mcC}_{\mcA}\sqrt{\frac{\beta_t}\mu}\,
\frac{\beta_t^2\norm{x^s-x^{s-1}}^2}{\epsilon^2}.
\end{equation}
Each such trial follows a distinct accepted trial. By~\eqref{eq:ncc-accepted-energy},~\eqref{eq:ncc-search-bound},~\eqref{eq:ncc-first-trial-evaluations}, and~\eqref{eq:ncc-later-trial-evaluations}, summing the trial bounds yields the following respective bounds for the numbers of evaluations of $h$, $\nabla h$, $\prox_p$, and $\prox_q$ in this group:
\begin{equation}\label{eq:ncc-group-evaluations}
\begin{aligned}[b]
&\widetilde{\mcC}_{\mcA}\sqrt{\frac{\beta_t}\mu}
\left(22+\frac{260L_{\nabla h}\Delta_0}{\epsilon^2}+\frac{36L_{\nabla h}D_{\mathbf y}}\epsilon\right)\le2^9\widetilde{\mcC}_{\mcA}\sqrt{\frac{\beta_t}\mu}
\left(1+\frac{L_{\nabla h}\Delta_0}{\epsilon^2}+\frac{L_{\nabla h}D_{\mathbf y}}\epsilon\right).
\end{aligned}
\end{equation}

Let $s_1<\cdots<s_J$ be the first indices of these groups. By step~7, $\beta_{s_{j+1}}=2\beta_{s_j}$ for every $1\le j<J$. Using this,~\eqref{eq:ncc-search-bound}, and $\mu=\epsilon/(32D)$, we obtain
\[
\begin{aligned}
\sum_{j=1}^J\sqrt{\frac{\beta_{s_j}}\mu}
&=\sqrt{\frac{\beta_{s_J}}\mu}\sum_{j=1}^J2^{-(J-j)/2}
\le(2+\sqrt2)\sqrt{\frac{2L_{\nabla h}}\mu}
=8(2+\sqrt2)\sqrt{\frac{L_{\nabla h}D}\epsilon}.
\end{aligned}
\]
Summing~\eqref{eq:ncc-group-evaluations} and adding the two gradient evaluations and one evaluation of each proximal mapping in step~2, we obtain that the numbers of evaluations of $h$, $\nabla h$, $\prox_p$, and $\prox_q$ in the run are bounded, respectively, by
\begin{equation}\label{eq:ncc-run-cost}
2+2^{12}(2+\sqrt2)\widetilde{\mcC}_{\mcA}\sqrt{\frac{L_{\nabla h}D}\epsilon}
\left(1+\frac{L_{\nabla h}\Delta_0}{\epsilon^2}+\frac{L_{\nabla h}D_{\mathbf y}}\epsilon\right).
\end{equation}
With $D=2D_{\mathbf y}$, Lemma~\ref{lem:ncc-run-cost} excludes failure, so Algorithm~\ref{alg:ncc} returns an $\epsilon$-stationary point of~\eqref{eq:ncc}. Since $D_{\mathbf y}\ge1$ and $\epsilon<32L_{\nabla h}$,~\eqref{eq:ncc-run-cost} implies~\eqref{eq:ncc-cost}. Hence the conclusion follows.
\end{proof}

\begin{proof}[Proof of Theorem~\ref{thm:pf-ncc}]
We first prove by induction that every reached phase $k$ satisfies $D_k\le4D_{\mathbf y}$ and its call to Algorithm~\ref{alg:ncc} terminates. Since $D_1=1\le D_{\mathbf y}$, the claim holds for $k=1$ by Lemma~\ref{lem:ncc-run-cost}. Suppose it holds at a reached phase $k$. If phase $k+1$ is reached, the call at phase $k$ returned failure. By Lemma~\ref{lem:ncc-run-cost}, $D_k<2D_{\mathbf y}$ and hence $D_{k+1}=2D_k<4D_{\mathbf y}$. Applying the same lemma at phase $k+1$ proves the induction claim.

Since $D_k=2^{k-1}$ and Lemma~\ref{lem:ncc-run-cost} excludes failure when $D_k\ge2D_{\mathbf y}$, only finitely many phases are reached. Let $K$ be the last phase. By Lemma~\ref{lem:ncc-run-cost}, its output is an $\epsilon$-stationary point of~\eqref{eq:ncc}. Moreover, $D_K\le4D_{\mathbf y}$, so
\[
\sum_{k=1}^K\sqrt{D_k}
=\sqrt{D_K}\sum_{k=1}^K2^{-(K-k)/2}
\le(2+\sqrt2)\sqrt{D_K}\le(2+\sqrt2)\sqrt{4D_{\mathbf y}}.
\]
Every phase starts from $(x^0,y^0)$ with the same tolerance, so $\Delta_0$ is unchanged. Since $D_k\ge1$ and $\epsilon<32L_{\nabla h}$, summing~\eqref{eq:ncc-run-cost} with $D=D_k$ and using the last inequality yields~\eqref{eq:pf-ncc-cost}, including the initialization in every phase. Hence the conclusion follows.
\end{proof}

\end{document}